\documentclass[pdflatex,sn-basic]{sn-jnl}

\usepackage{graphicx}%
\usepackage{multirow}%
\usepackage{amsmath,amssymb,amsfonts}%
\usepackage{amsthm}%
\usepackage{mathrsfs}%
\usepackage[title]{appendix}%
\usepackage{xcolor}%
\usepackage{textcomp}%
\usepackage{manyfoot}%
\usepackage{booktabs}%
\usepackage{enumitem}
\usepackage{siunitx}
\usepackage{hyphenat}
\usepackage{algorithm}%
\usepackage{algorithmicx}%
\usepackage{algpseudocode}%
\usepackage{listings}%
\usepackage[justification=centering]{caption}
\usepackage{tikz-cd}

\newcommand{\Dgm}{\mathrm{Dgm}}
\newcommand{\Pers}{\mathrm{Pers}}

\theoremstyle{thmstyleone}%
\newtheorem{theorem}{Theorem}%  meant for continuous numbers
\newtheorem{proposition}[theorem]{Proposition}% 
\newtheorem{lemma}[theorem]{Lemma}%
\newtheorem{corollary}[theorem]{Corollary}%

\theoremstyle{thmstyletwo}%
\newtheorem{example}{Example}%
\newtheorem{remark}{Remark}%
\newtheorem{assumption}{Assumption}%

\theoremstyle{thmstylethree}%
\newtheorem{definition}{Definition}%

\begin{document}

\title[Persistence-Based Statistics for Detecting Structural Changes in High-Dimensional Point Clouds]%
      {Persistence-Based Statistics for Detecting Structural Changes in High-Dimensional Point Clouds}

\author*[1]{\fnm{Toshiyuki}\sur{Nakayama}}\email{nakayama\_t25@shunan-u.ac.jp}

\affil[1]{%
  \orgdiv{Faculty of Information Science, }%
  \orgname{Shunan University, }%
  \orgaddress{%
    \street{843-4-2, Gakuendai, }%
    \city{Shunan City, }%
    \state{Yamaguchi, }%
    \postcode{745-8566, }%Springer 
    \country{Japan}%
  }%
}

\abstract{
We study the probabilistic behavior of persistence-based statistics and propose a novel nonparametric framework for detecting structural changes in high-dimensional random point clouds. We establish moment bounds and tightness results for classical persistence statistics—total and maximum persistence—under general distributions, with explicit variance-scaling behavior derived for Gaussian mixture models. Building on these results, we introduce a bounded and normalized statistic based on persistence landscapes combined with the Jensen–Shannon divergence, and we prove its Hölder continuity with respect to perturbations of the input point clouds. The resulting measure is stable, scale- and shift-invariant, and well suited for finite-sample nonparametric inference via permutation testing. An illustrative numerical study using dynamic attribute vectors from decentralized governance data demonstrates the practical applicability of the proposed method. Overall, this work provides a statistically rigorous and computationally stable approach to change-point detection in complex, high-dimensional data.
}


\maketitle

% -----------------------------------------------
\section{Introduction}
\label{introduction}
% -----------------------------------------------
Statistical detection of structural changes in high-dimensional data
is a fundamental problem in modern data analysis, with applications
ranging from complex networks to time-evolving multivariate systems.
When data are represented as point clouds, such changes often manifest
as alterations in their geometric or topological organization.
Designing nonparametric and computationally stable procedures that can
capture such changes in a distribution-free manner remains a challenging task.

Topological data analysis (TDA) provides a powerful framework for
quantifying geometric and topological structure in point cloud data.
In many real-world systems, however, the underlying data-generating
process evolves over time, leading to gradual or abrupt changes in the
associated topology.
Understanding how persistence statistics behave under such variability
—and how they can be used for statistically rigorous inference—remains
an open problem at the interface of statistics, probability theory,
and computational topology.

The present work addresses this problem from a theoretical standpoint. 
We investigate the probabilistic properties of persistence statistics, 
establish explicit moment bounds and scaling relations under general and Gaussian-mixture settings, 
and develop a new persistence-based divergence for change detection. 
Specifically, our contributions are threefold:

\begin{enumerate}[label=(\arabic*)]
\item \textbf{Moment and tightness results for random persistence:} 
We derive upper bounds for the \(p\)-th moments of total and maximum persistence 
under minimal distributional assumptions, and we show tightness of the normalized statistics 
under variance scaling. 
These results provide new quantitative insight into the stochastic behavior of persistence features.

\item \textbf{A new persistence-landscape-based divergence:} 
We define a scale- and shift-invariant statistic combining persistence landscapes with 
the Jensen--Shannon distance, and we prove its H\"older continuity 
with respect to perturbations of point clouds. 
This yields a stable metric suitable for nonparametric statistical inference.

\item \textbf{Illustrative application to dynamic point clouds:} 
We demonstrate the proposed framework on time-evolving high-dimensional attribute vectors, 
detecting regime shifts in a decentralized voting dataset. 
The example highlights how theoretical stability results translate into practical robustness.
\end{enumerate}

Our study thus contributes to the ongoing effort to build 
a probabilistic theory of persistent homology, 
linking classical stochastic analysis with modern TDA. 
Beyond the illustrative example, the proposed framework applies generally 
to random point clouds and dynamic metric spaces, 
providing a mathematically grounded approach to topological change detection.

% -----------------------------------------------
\section{Background and Related Work}\label{background}
% -----------------------------------------------

\subsection{High-Dimensional and Dynamic Point Clouds}
Modern data often arise as collections of high-dimensional points whose underlying 
geometry and topology evolve over time. 
Examples include dynamic networks, evolving manifolds, 
and temporal collections of attribute vectors in complex systems. 
Analyzing such data requires statistical methods that are sensitive to 
nonlinear structural changes yet robust to noise and scale variation.
A central challenge is to detect significant shifts in the topological organization 
of these point clouds---that is, changes in their connectivity or higher-dimensional 
features---without imposing strong distributional assumptions.

Our motivation stems from the general problem of monitoring \emph{random point clouds} 
generated by stochastic processes whose dispersion or correlation structure may vary over time.
While specific applications may include social, biological, or economic systems,
the focus of this study is on the mathematical and statistical foundations 
that enable such topological monitoring in a general probabilistic setting.

% -------------------------------------
\subsection{Topological Data Analysis and Related Work}
Topological data analysis (TDA) provides a powerful framework 
for extracting geometric and topological information from point clouds.
Through a family of simplicial complexes constructed across 
varying proximity thresholds---a process known as \emph{filtration}---%
TDA records the birth and death of topological features in a 
\emph{persistence diagram}, a multiset of intervals representing the data’s shape.
The theoretical foundations of persistent homology were established 
by \cite{carlsson2009} and \cite{edelsbrunner2010}, 
and have since found numerous applications in the sciences~\cite{lum2013, ghrist2008}.

To enable statistical analysis, several stable vectorizations of persistence diagrams 
have been proposed. 
The \emph{persistence landscape} introduced by \cite{bubenik2015} represents a diagram as a collection of functions, 
while the \emph{persistence image} introduced by \cite{adams2017} provides a smooth, fixed-length embedding. 
These representations form the basis of the emerging field of \emph{statistical TDA}, 
which connects persistent homology with classical inference and probability theory. 

Recent studies have explored persistence-based indicators of change or instability in complex systems.
For instance, \cite{gidea2018} applied persistent homology to 
dynamic correlation networks, while \cite{islambekov2019} 
combined persistence features with nonparametric change-point algorithms for 
environmental data. 
More recently, \cite{yao2025} analyzed time-varying financial networks 
through persistent homology to detect regime shifts. 
These works highlight the potential of topological methods for detecting 
structural transitions, yet a unified \emph{theoretical} and \emph{nonparametric} framework 
for this purpose has not been established.

From a statistical perspective, many existing persistence-based approaches
focus primarily on feature extraction or empirical summaries, while
questions of finite-sample behavior, distribution-free inference, and
computational stability under data perturbations remain comparatively less explored.

The present study addresses this gap by developing a mathematically grounded, 
persistence-based framework for detecting structural changes in random point clouds.
We establish moment and tightness bounds for classical persistence statistics 
and prove stability and continuity results for a newly defined landscape-based divergence.

% -----------------------------------------------
\section{Persistence Statistics and Moment Bounds}
% -----------------------------------------------
This section adopts a standard formal framework for the persistent homology
of finite point clouds, where persistence diagrams are treated as
diagram-valued mappings induced from the data through filtration
constructions.
The overall setup—including the representation of a point cloud as a
finite subset of a metric space and the use of persistence-based
real-valued quantities derived from filtrations—follows the general
formulation introduced by \cite{carriere2021}.
Within this framework, we focus on probabilistic moment bounds and
scaling properties of persistence statistics.
\subsection{Persistence Statistics}
\label{sec:general_preliminaries}
Let \(d, N \in \mathbb{N}\) and let \(x = (x_1, \ldots, x_N) \in \mathcal{X}^N\) 
be a finite point cloud in \(\mathcal{X}:=\mathbb{R}^d\), where each element of \(\mathcal{X}\) represents an attribute vector. 
The set of all filtrations on the full simplex \(K = 2^{\{1,\ldots,N\}} \setminus \{\emptyset\}\) is 
\[
\mathrm{Filt}_K =
\left\{
\Phi = (\Phi_\sigma)_{\sigma \in K} \in \mathbb{R}^{\lvert K \rvert} \,:\, \,
\sigma_1 \subset \sigma_2 \implies \Phi_{\sigma_1} \le \Phi_{\sigma_2}
\right\}
\subset\mathbb{R}^{2^N - 1}.
\]
Here, \(\lvert K \rvert\) denotes the cardinality of \(K\).

Let \(\Dgm\) denote the space of persistence diagrams, i.e. locally finite multisets 
on \(\{(b,d)\in\overline{\mathbb{R}}^2: b<d\}\) with the convention that the diagonal 
\(\{(t,t):t\in\overline{\mathbb{R}}\}\) has infinite multiplicity. 
We endow \(\Dgm\) with the bottleneck distance \(d_B\) (Appendix~\ref{sec:bottleneck_distance}).
The map
\[
\Pers_\ell:\mathrm{Filt}_K\to\Dgm
\]
assigns to each filtration its persistence diagram in homological degree \(\ell\in\mathbb{N}_0\); 
this is the \emph{persistence map}.
(Basic notions of simplicial complexes, homology, and persistence diagrams 
are recalled in Appendix~\ref{sec:tda_basics}.)

For any \(D \in \Dgm\), define the \emph{total persistence} in degree \(\ell\in\mathbb{N}_0\) \(E_{\ell,\mathrm{total}} \colon \Dgm \to \mathbb{R}\) by 
\[
E_{\ell,\mathrm{total}}(D) = \sum_{(b, d) \in D,\, d < \infty} (d - b),
\]
where the sum is taken with multiplicities.
Similarly, 
for any \(D \in \Dgm\), define the \emph{maximum persistence} in degree \(\ell\in\mathbb{N}_0\) \(E_{\ell,\mathrm{max}} \colon \Dgm \to \mathbb{R}\) by 
\[
E_{\ell,\mathrm{max}}(D) = \max_{(b, d) \in D,\, d < \infty} (d - b).
\]
Both quantities are finite for finite diagrams, and intuitively measure the 
total and maximal lifespans of \(\ell\)-dimensional topological features.

\paragraph{Vietoris--Rips filtration.}
In what follows we work with the Vietoris--Rips filtration map \(\Phi \colon \mathcal{X}^N \to \mathrm{Filt}_K\), defined by 
\begin{equation}
\label{eq:vr_filtration}
\Phi_\sigma(x)
  := \max_{i,j\in\sigma}\|x_i-x_j\|, \qquad \sigma\in K,
\end{equation}
where \(\|\cdot\|\) denotes the Euclidean norm.  
This yields the \emph{Vietoris–Rips persistence diagram map} \(\mathcal{D}_\ell\colon\mathcal{X}^N\to\Dgm\) in degree \(\ell\), defined by
\[
\mathcal{D}_\ell(x) := \Pers_\ell(\Phi(x)) \in \Dgm,\quad x\in\mathcal{X}^N.
\]

\begin{remark}
Here we adopt the convention that the Vietoris–Rips filtration parameter 
is given by the pairwise distance itself, rather than by half the distance. 
This choice is consistent with standard implementations such as Gudhi.
\end{remark}

% --- Example --------------------------------------
\begin{example}[Three-point configuration]
Let \(x=(x_1,x_2,x_3)\in(\mathbb{R}^d)^3=\mathcal{X}^3\),
and let \(s_{(1)}(x)\le s_{(2)}(x)\le s_{(3)}(x)\) be the ordered pairwise distances
\(\{\|x_1-x_2\|,\|x_1-x_3\|,\|x_2-x_3\|\}\).
The Vietoris–Rips filtration values are
\begin{align*}
&\Phi_{\{i\}}(x)=0 \quad (i=1,2,3),\\
&\Phi_{\{i,j\}}(x)=\|x_i-x_j\| \quad (1\le i<j\le 3),\\
&\Phi_{\{1,2,3\}}(x)=s_{(3)}(x).
\end{align*}
The resulting diagrams are given by
\begin{align*}
\mathcal{D}_0(x) &=
\{(0, s_{(1)}(x))\} \uplus
\{(0, s_{(2)}(x))\} \uplus
\{(0, \infty)\},\\
\mathcal{D}_\ell(x) &= \emptyset,\qquad \ell\ge1.
\end{align*}
Here, the symbol \(\uplus\) denotes the disjoint union of finite multisets,
so that multiplicities of identical birth–death pairs are added accordingly.
These represent the creation and merging of connected components in \(H_0\).
Hence
\[
E_{0,\mathrm{total}}(\mathcal{D}_0(x))
  =s_{(1)}(x)+s_{(2)}(x),\qquad
E_{0,\mathrm{max}}(\mathcal{D}_0(x))
  =s_{(2)}(x).
\]
\end{example}

% -----------------------------------------------
\subsection{Continuity and Regularity of Persistence Statistics}
In what follows, to evaluate quantities such as expectations of powers of persistence statistics (total and maximum persistence) (see Section~\ref{sec:general_preliminaries}), we first confirm basic regularity (continuity and semicontinuity, hence measurability):

\begin{proposition}[Basic Regularity]
\label{prop:continuous}
Under the setting of Section~\ref{sec:general_preliminaries}, the following follows. 
\begin{enumerate}
\item[(i)] For each \(\ell\), map \(\Pers_\ell:\mathrm{Filt}_K\to \Dgm\) is continuous (with respect to \(d_B\)).
\item[(ii)] For each \(\ell\), map \(\mathcal{D}_\ell=\Pers_\ell\circ\Phi:\mathcal{X}^N\to\Dgm\) is continuous.
\item[(iii)] For each \(\ell\), map \(E_{\ell,\mathrm{max}}:\Dgm\to\mathbb{R}\) is continuous and \(E_{\ell,\mathrm{total}}:\Dgm\to\mathbb{R}\) is lower semicontinuous.
\end{enumerate}
\end{proposition}

\begin{proof}
(i) The continuity of each \(\Pers_\ell\) follows from the stability of persistence diagrams 
with respect to the bottleneck distance, as shown by \cite{cohen2007}. 

(ii) The filtration map \(\Phi:\mathcal{X}^N\to \mathrm{Filt}_K\) in \eqref{eq:vr_filtration} 
is continuous (it is a composition of continuous maps of Euclidean distances). 
Thus \(\mathcal{D}_\ell=\Pers_\ell\circ\Phi\) is continuous as a composition of continuous maps.

(iii) This is already known; see, e.g., \cite{cohen2007}. 
\end{proof}

We may now define the following maps:
\[
\begin{aligned}
\mathcal{L}_{\ell,\mathrm{total}} &:= E_{\ell,\mathrm{total}} \circ \mathcal{D}_\ell : \mathcal{X}^N \to \mathbb{R}, \\
\mathcal{L}_{\ell,\mathrm{max}}   &:= E_{\ell,\mathrm{max}}   \circ \mathcal{D}_\ell : \mathcal{X}^N \to \mathbb{R}.
\end{aligned}
\]

These maps assign to each point cloud the total or maximum bar length of its persistence diagram.
By Proposition~\ref{prop:continuous}, they are at least lower semicontinuous  
and hence measurable. 
In the next section, we endow the point cloud with a probability distribution and evaluate the 
\(p\)-th moments of these maps.

% -----------------------------------------
\subsection{Moment Bounds for General Distributions}

Let \((\Omega, \mathcal{F}, \mathbb{P})\) be a probability space. 
We define the random vector
\[
  X = \bigl( X_1, \ldots, X_N \bigr) : 
  \Omega \to \mathcal{X}^N,
\]
where \(X_1,\ldots,X_N\) are \(\mathcal{X}\)-valued random variables 
having the same marginal distribution under \(\mathbb{P}\), 
but not necessarily independent.
(Here, each \(X_i\in\mathcal{X}\) represents one attribute vector.)
Throughout, \(\mathbb{E}[\cdot]\) denotes the expectation with respect to \(\mathbb{P}\).

Throughout this section, let \(\ell \in \mathbb{N}_0\) be arbitrary and set
\[
\mathscr{L}_\ell :=
\{
\mathcal{L}_{\ell, \mathrm{total}}, \mathcal{L}_{\ell, \mathrm{max}}\}.
\]
We write \(\mathcal{L} \in \mathscr{L}_\ell\) for a generic element of this set.
We first provide an upper bound on the \(p\)-th moment of \(\mathcal{L}(X)\) without assuming any specific distribution for \(X\).

\begin{theorem}[Moment Bounds of Persistence Statistics for General Distributions]
\label{thm:main_estimation}
For any \(p \ge 1\), the following inequality holds:
\[
\mathbb{E}\left[\mathcal{L}(X)^p\right]
\le C_{\ell,p}^{\mathrm{VR}}(N)\,\mathbb{E}\left[\Vert X_1\Vert^p\right],
\]
where \(\|\cdot\|\) is the Euclidean norm and 
\[
n_\ell(N):=\min\!\left\{
\binom{N}{\ell+1},\, 
\binom{N}{\ell+2}
\right\}
,\quad
C_{\ell,p}^{\mathrm{VR}}(N):=2^{p}N\left(n_\ell(N)\right)^p.
\]
\end{theorem}
\begin{proof}
Since 
\[
\mathcal{L}_{\ell,\mathrm{max}}(x) \le \mathcal{L}_{\ell,\mathrm{total}}(x),\quad x\in\mathcal{X}^N,
\]
it suffices to prove the theorem for \(\mathcal{L}_{\ell,\mathrm{total}}\) only.

Each finite bar in degree~\(\ell\) corresponds to a unique pairing between 
an \(\ell\)-simplex and an \((\ell+1)\)-simplex in the standard 
matrix-reduction algorithm~\cite{edelsbrunner2002, edelsbrunner2010}, 
implying that the number of finite bars is bounded by
\[
\#\{\text{finite bars in }H_\ell\}
\;\le\;n_\ell(N).
\]

By H\"older’s inequality we have
\begin{equation}
\label{eq:total_persistence_estimation}
\mathbb{E}\left[\mathcal{L}_{\ell,\mathrm{total}}(X)^p\right]
\le n_\ell(N)^{p-1} I,
\end{equation}
where we set
\[
I :=
\mathbb{E}\left[
\sum_{\stackrel{(b, d) \in \Pers_\ell(\Phi(X))}{d < \infty}}
(d - b)^p \right].
\]

In the case of Vietoris-Rips filtration, the filtration value of the vertex is 0 (see the definition~\eqref{eq:vr_filtration}), so we have
\begin{align*}
I &\le n_\ell(N)
\mathbb{E}\left[
\max_{\sigma_1, \sigma_2 \in K}
\left|\Phi_{\sigma_1}(X) - 
\Phi_{\sigma_2}(X)\right|^p\right] \\
& = n_\ell(N) \mathbb{E}\left[
\max_{\sigma \in K}
\left|\Phi_{\sigma}(X)\right|^p\right].
\end{align*}

Again, from the definition of the Vietoris–Rips filtration in \eqref{eq:vr_filtration}, we obtain
\begin{align*}
    I & \le n_\ell(N)\, \mathbb{E}\left[
\max_{j_1,j_2\in\{1,\ldots,N\}}
\left\Vert X_{j_1} - 
X_{j_2}\right\Vert^p\right] \\
& \le n_\ell(N)\, \mathbb{E}\left[
\max_{j_1,j_2\in\{1,\ldots,N\}}
\left(
\left\Vert X_{j_1} \right\Vert
+\left\Vert X_{j_2} \right\Vert
\right)^p\right] \\
& \le 2^{p}n_\ell(N)\, \mathbb{E}\left[
\max_{j\in\{1,\ldots,N\}}
\left\Vert X_{j}\right\Vert^p\right] \\
& \le 2^{p}n_\ell(N)\, \sum_{j=1}^N 
\mathbb{E}\left[
\left\Vert X_{j}\right\Vert^p\right]
=2^{p}Nn_\ell(N)\,\mathbb{E}\left[\left\Vert X_{1}\right\Vert^p\right].
\end{align*}
Combining the equation~\eqref{eq:total_persistence_estimation} leads to the conclusion.
\end{proof}

% -----------------------------------------
\subsection{Moment Bounds and Tightness under GMMs}
\label{sec:gmm}

In the previous section, we established Theorem~\ref{thm:main_estimation}, 
which provides moment bounds for persistence statistics under general distributions. 
Here we focus on the special case where the data points are drawn from a Gaussian mixture model (GMM).

\begin{remark}
It is well known that GMMs are universal approximators of probability densities on \(\mathbb{R}^d\).
That is, any probability distribution can be approximated arbitrarily well by a finite mixture of Gaussian components 
(see, e.g., \cite{mclachlan2000, li2000}).
This property justifies the use of GMMs as flexible models for complex data distributions such as attribute vectors.
\end{remark}

For integers \(d,K\ge 1\), consider a \(d\)-dimensional GMM with
mixture weights \((\pi_1,\ldots,\pi_K)\), mean vectors \(\mu_1,\ldots,\mu_K\in\mathbb{R}^d\),
and covariance matrices \(\eta\Sigma_1,\ldots,\eta\Sigma_K\), where \(\eta>0\) is a common
variance–scaling factor and \(\Sigma_1,\ldots,\Sigma_K\in\mathbb{S}_+^d\) are the
unscaled covariances. Here
\[
\mathbb{S}_+^d := \{A\in\mathbb{R}^{d\times d}: A=A^\top,\ A\succeq 0\}
\]
denotes the cone of symmetric positive semidefinite (PSD) matrices. We use the
Loewner order: \(A\succeq B\) means \(A-B\) is PSD, and \(A\succ 0\) denotes positive
definite (PD). Allowing \(\Sigma_k\succeq 0\) permits singular covariances.
Let \(P_\eta\) denote the corresponding probability distribution on \(\mathcal{X}=\mathbb{R}^d\).

Let \((\Omega,\mathcal{F},\mathbb{P})\) be a probability space.
On this probability space, we define the random vector
\[
  X^{(\eta)} = \bigl( X_1^{(\eta)}, \ldots, X_N^{(\eta)} \bigr) : 
  \Omega \to \mathcal{X}^N,
\]
where \(X_1^{(\eta)}, \ldots, X_N^{(\eta)}\) are 
\(\mathcal{X}\)-valued random variables 
sharing the same marginal distribution \(P_\eta\);
they are not required to be independent unless stated otherwise.

Throughout, we use the notation \(N(\mu_k,\Sigma_k)\) to denote the Gaussian law 
with mean \(\mu_k\) and covariance \(\Sigma_k\in\mathbb{S}_+^d\), allowing for singular \(\Sigma_k\). 

% ----------------------------------------------
\begin{remark}[Gaussian distributions with possibly singular covariance matrices]
\label{rem:possibly_singular_covariances}

Let
\[
\Sigma_k \;=\; U\,\mathrm{diag}(\lambda_1,\ldots,\lambda_d)\,U^\top,
\qquad U^\top U=I,\ \lambda_i\ge0.
\]
(Recall that every real symmetric PSD matrix admits such a decomposition, with nonnegative eigenvalues.)
The law \(N(\mu_k,\Sigma_k)\) can be specified in any of the following equivalent ways:
\begin{description}
    \item[Principal square root.] 
    The matrix \(\Sigma_k\) admits the unique symmetric positive semidefinite square root, called the \emph{principal square root}, denoted by \(\Sigma_k^{1/2}\), and given by 
    \[
    \Sigma_k^{1/2} \;=\; U\,\mathrm{diag}\big(\sqrt{\lambda_1},\ldots,\sqrt{\lambda_d}\big)\,U^\top .
    \]
    (Although the eigenbasis \(U\) need not be unique, the product above is unique.)
    If \(Z_d\sim N(0,I_d)\), then
    \[
    \mu_k+\Sigma_k^{1/2} Z_d \;\sim\; N(\mu_k,\Sigma_k)
    \quad
    (\operatorname{Cov}(\Sigma_k^{1/2}Z_d)=\Sigma_k^{1/2}(\Sigma_k^{1/2})^\top=\Sigma_k).
    \]
    \item[Thin factor (rank–\(r_k\) representation).]
    Let \(r_k=\mathrm{rank}(\Sigma_k)\) and collect the eigenvectors corresponding to the
    \emph{positive} eigenvalues into \(U_r=[u_i:\lambda_i>0]\in\mathbb{R}^{d\times r_k}\),
    and set \(\Lambda_r=\mathrm{diag}(\lambda_i:\lambda_i>0)\). Define
    \[
    B_k \;:=\; U_r\,\Lambda_r^{1/2}\in\mathbb{R}^{d\times r_k}.
    \]
    Then \(\Sigma_k=B_k B_k^\top\), and for \(Z_{r_k}\sim N(0,I_{r_k})\) we have
    \[
    \mu_k + B_k Z_{r_k} \;\sim\; N(\mu_k,\Sigma_k)
    \quad(\operatorname{Cov}(B_k Z_{r_k})=B_k B_k^\top=\Sigma_k).
    \]
    Note that \(B_k\) is not unique: \(B_k Q\) with any orthogonal \(Q\in\mathbb{R}^{r_k\times r_k}\)
    yields the same covariance.
    \item[The characteristic function.]
    Use the characteristic
function \(\exp\big(i t^\top\mu_k-\tfrac12 t^\top\Sigma_k t\big)\) for \(t\in\mathbb{R}^d\).
\end{description}
\end{remark}

\begin{remark}[Key Features of Our Model]
\label{rem:key_features_of_gmm}
For the modeling described above, we note the following points.
\begin{enumerate}[label=(\roman*)]
\item
Throughout we work with (possibly singular) symmetric positive semidefinite covariances \(\Sigma_k\in\mathbb{S}_+^d\) and do not assume strict positive definiteness. This modeling choice accommodates effectively low-rank components that arise in practice--for instance, in the voting context, this occurs under strong substitutability among candidates, where voters treat groups of candidates as equivalent, causing their allocated scores (ratings) to co-move (e.g., approximately constant-sum) and induce near-linear dependencies. While many software implementations insert a small diagonal loading for numerical stability (e.g., \(\Sigma_k\mapsto\Sigma_k+\lambda_0 I_d\) with a small \(\lambda_0>0\)), our results are formulated for \(\Sigma_k\succeq0\) and remain valid in the vanishing regularization limit \(\lambda_0\downarrow 0\).

\item
As usual, the mixture is identifiable only up to permutation of component labels; all results below are invariant under such permutations.

\item
If \(\Sigma_k\) is singular, the \(k\)-th Gaussian component \(\mathcal N(\mu_k,\Sigma_k)\) is supported on the affine subspace \(\mu_k+\operatorname{Im}(\Sigma_k)\) (equivalently, \(\mu_k+\operatorname{Im}(\Sigma_k^{1/2})\)). Therefore
\[
\operatorname{supp}(P_\eta)=\overline{\bigcup_{k:\,\pi_k>0}\bigl(\mu_k+\operatorname{Im}(\Sigma_k)\bigr)}.
\]
In particular, \(\operatorname{supp}(P_\eta)=\mathbb{R}^d\) if at least one component has \(\Sigma_k\succ0\).
\end{enumerate}
\end{remark}

% -----------------------------------------------------
We introduce a constant that will be used in subsequent estimates.
\begin{definition}[Gaussian Moment Constant]
\label{def:gaussian_moment_constant}
Let \(p\ge1\). Define
\[
  \kappa_{d,p} := \mathbb{E}\bigl[\|Z\|^p\bigr]
  = 2^{p/2}\,\frac{\Gamma\left(\tfrac{d+p}{2}\right)}{\Gamma\left(\tfrac{d}{2}\right)},
\]
where \(Z \sim N(0, I_d)\) and \(\|\cdot\|\) is the Euclidean norm.
\end{definition}

\begin{remark}
The closed form in Definition~\ref{def:gaussian_moment_constant} follows
from the fact that \(\|Z\|\) has the \(\chi_d\) distribution, i.e., the distribution
of the square root of a \(\chi^2_d\) random variable. The \(p\)-th moment of
the \(\chi_d\) distribution is well known; see, e.g.,
\citet[Sec.~2.3]{vershynin2018} or \citet[Ch.~1]{muirhead1982},
yielding
\[
  \mathbb{E}[\|Z\|^p]
  = 2^{p/2}\,\frac{\Gamma\left(\tfrac{d+p}{2}\right)}{\Gamma\left(\tfrac{d}{2}\right)}.
\]
\end{remark}

% -----------------------------------------
Using this constant, we can bound the \(p\)-th moments of  general Gaussian vectors as follows.
\begin{lemma}[Moment Bound for Gaussian Vectors]
\label{lem:gaussian_moment}
Let \(p\ge1\) and let \(\mu\in\mathbb{R}^d, \Sigma\in\mathbb{S}_+^d\). If
\(Y\sim N(\mu,\Sigma)\), then
\[
\mathbb{E}\left[\|Y\|^p\right]\;\le\; 
2^{p-1}\left\{\|\mu\|^p+
\kappa_{d,p}\, \bigl(\mathrm{tr}\,\Sigma\bigr)^{p/2}\right\}.
\]
\end{lemma}
\begin{proof}
Let \(Z\sim N(0,I_d)\) and write \(Y=\mu+\Sigma^{1/2}Z\), where \(\Sigma^{1/2}\) is the principal square root defined in Remark~\ref{rem:possibly_singular_covariances}. 
Then we have
\begin{align*}
\|Y\|&\le\|\mu\|+\|\Sigma^{1/2}Z\| \\
&\le \|\mu\|+\|\Sigma^{1/2}\|_{\mathrm F}\,\|Z\|
= \|\mu\|+(\mathrm{tr}\,\Sigma)^{1/2}\|Z\|,
\end{align*}
where \(\|\cdot\|_{\mathrm F}\) is the Frobenius norm on matrices, defined by
\(\|A\|_{\mathrm F}:=\sqrt{\mathrm{tr}(A^\top A)}\).

From H\"{o}lder's inequality, we have
\begin{align*}
\mathbb{E}\left[\|Y\|^p\right] 
&\le
2^{p-1}\left(\|\mu\|^p+(\mathrm{tr}\,\Sigma)^{p/2}\mathbb{E}\left[\|Z\|^p\right]\right) \\
&= 2^{p-1}\left\{\|\mu\|^p+\kappa_{d,p}(\mathrm{tr}\,\Sigma)^{p/2}\right\}.
\end{align*}
\end{proof}

% -----------------------------------------
\begin{lemma}[Norm Estimate for a Point in the Point Cloud]
\label{lem:random_points_norm}
For any \(p \ge 1\), there exist constants \(M_p, V_p > 0\) such that
\[
\mathbb{E}
\left[
\Vert X_{j}^{(\eta)} \Vert^p \right]
\le 
M_p + V_p \eta^{p/2},\quad 
j=1, \ldots, N,\ \eta>0.
\]
Here we can choose as follows: 
\[
M_p=2^{p-1}\sum_{k=1}^K\pi_k 
\Vert \mu_k\Vert^p, \quad
V_p=2^{p-1}\kappa_{d,p}\sum_{k=1}^K\pi_k 
(\mathrm{tr}\,\Sigma_k)^{p/2}.
\]
\end{lemma}

\begin{proof}
From Lemma~\ref{lem:gaussian_moment}, we have
\begin{align*}
\mathbb{E}\left[
\Vert X_{j}^{(\eta)} \Vert^p \right] 
& =
\sum_{k=1}^K\pi_k
\mathbb{E}\left[\Vert Y_k^{(\eta)}\Vert^p\right]
\quad(Y_k^{(\eta)}\sim\mathcal{N}(\mu_k,\eta\Sigma_k)) \\
& \le
2^{p-1}\sum_{k=1}^K\pi_k \left(
\Vert \mu_k\Vert^p+
\kappa_{d,p}(\mathrm{tr}\,\Sigma_k)^{p/2} \eta^{p/2}
\right)
=M_p+V_p\eta^{p/2}.
\end{align*}
This completes the proof.
\end{proof}

% -----------------------------------------
We now evaluate persistence statistics under variance scaling for GMMs.
\begin{theorem}[Moment Bounds of Persistence Statistics for GMMs]
\label{thm:lp_estimation_gmm}
For any \(p \ge 1\) and any \(\mathcal{L}\in\mathscr{L}_\ell\), there exist constants \(C^{(0)}_{\ell,p}(N), C^{(1)}_{\ell,p}(N) > 0\)  such that the following inequality holds:
\[
\mathbb{E}\left[\mathcal{L}(X^{(\eta)})^p\right]
\le C^{(0)}_{\ell,p}(N) + C^{(1)}_{\ell,p}(N) \eta^{p/2},\quad \eta>0.
\]
Here we can choose as follows: 
\[
C^{(0)}_{\ell,p}(N)=M_p C_{\ell,p}^{\mathrm{VR}}(N), \quad
C^{(1)}_{\ell,p}(N)=V_p C_{\ell,p}^{\mathrm{VR}}(N),
\]
where \(C_{\ell,p}^{\mathrm{VR}}(N)\) is defined in Theorem~\ref{thm:main_estimation}.

In particular, for the same \(p\) we have
\[
\sup_{\eta\ge 1}\mathbb{E}\bigl[(\mathcal{L}(X^{(\eta)})/\sqrt{\eta})^p\bigr]\le C_{\ell,p}(N),
\]
where we set
\[
C_{\ell,p}(N):=C^{(0)}_{\ell,p}(N)+C^{(1)}_{\ell,p}(N).
\]
\end{theorem}
\begin{proof}
By Theorem~\ref{thm:main_estimation},
\(\mathbb{E}[\mathcal{L}(X^{(\eta)})^p]\le C_{\ell,p}^{\mathrm{VR}}(N)\,\mathbb{E}\bigl[\|X^{(\eta)}_1\|^p\bigr]\).
Lemma~\ref{lem:random_points_norm} yields
\(\mathbb{E}\bigl[\|X^{(\eta)}_1\|^p\bigr]\le M_p+V_p\,\eta^{p/2}\).
Combine them and set \(C^{(0)}_{\ell,p}(N)=M_p C_{\ell,p}^{\mathrm{VR}}(N)\),
\(C^{(1)}_{\ell,p}(N)=V_p C_{\ell,p}^{\mathrm{VR}}(N)\).
The “in particular” part follows by dividing by \(\eta^{p/2}\) and taking \(\sup_{\eta\ge1}\).
\end{proof}

% --------------------------------------------
By normalizing persistence statistics   
by the square root of the variance multiplier,  
we can ensure that the probability of leaving a certain compact set is sufficiently small,  
as shown below.

\begin{corollary}[Tightness of the Normalized Persistence Statistics]
\label{cor:tightness}
For any \(\mathcal{L}\in\mathscr{L}_\ell\), the family
\(\{\mathcal{L}(X^{(\eta)})/\sqrt{\eta}\}_{\eta\ge1}\) is tight on \(\mathbb{R}\).
In particular, for any \(\varepsilon>0\) there exists \(U_{\varepsilon,\ell,p}(N)>0\) such that
\[
  \sup_{\eta\ge1}\mathbb{P}\!\left(
  \frac{\mathcal{L}(X^{(\eta)})}{\sqrt{\eta}} > U_{\varepsilon,\ell,p}(N)\right)
  < \varepsilon,
\]
where we can choose
\[
U_{\varepsilon,\ell,p}(N)
:=\left(\frac{C_{\ell,p}(N)}{\varepsilon}\right)^{1/p},
\]
with \(C_{\ell,p}(N)\) defined in Theorem~\ref{thm:lp_estimation_gmm}.
\end{corollary}

\begin{proof}
By Theorem~\ref{thm:lp_estimation_gmm},
\(\sup_{\eta\ge1}\mathbb{E}\!\left[(\mathcal{L}(X^{(\eta)})/\sqrt{\eta})^p\right]
\le C_{\ell,p}(N)<\infty\) for some \(p\ge1\).
Markov's inequality yields
\(\sup_{\eta\ge1}\mathbb{P}(\mathcal{L}(X^{(\eta)})/\sqrt{\eta}>U)
\le C_{\ell,p}(N)/U^p\).
Choosing \(U_{\varepsilon,\ell,p}(N)=(C_{\ell,p}(N)/\varepsilon)^{1/p}\) completes the proof.
\end{proof}

% -------------------------------------

Since the family is tight, every sequence has a weakly convergent subsequence, as shown below.

\begin{corollary}
For any sequence \(\{\eta_n\}_{n\ge1}\subset[1,\infty)\), there exists a subsequence
\(\{\eta_{n_k}\}_{k\ge1}\) such that
\(\mathcal{L}(X^{(\eta_{n_k})})/\sqrt{\eta_{n_k}}\) converges in distribution
to some probability measure on \(\mathbb{R}\) as \(k\to\infty\).
\end{corollary}

\begin{proof}
By Corollary~\ref{cor:tightness} the family is tight; since \(\mathbb{R}\) is Polish,
Prokhorov's theorem gives relative compactness in the weak topology.
\end{proof}

% -------------------------------------
To quantify the scale of the statistic, we prove that the upper tail above a constant multiple of \(\sqrt{\eta}\) is uniformly small.

\begin{corollary}[Uniform polynomial tail at the \(\sqrt{\eta}\)-scale]
\label{cor:uniform_poly_tail}
Fix \(\mathcal{L}\in\mathscr{L}_\ell\) and \(p\ge1\). 
Let \(C_{\ell,p}(N)>0\) be the constant from
Theorem~\ref{thm:lp_estimation_gmm}.
Then for all \(t>0\),
\[
\sup_{\eta\ge1}
\mathbb{P}\!\left(
  \frac{\mathcal{L}(X^{(\eta)})}{\sqrt{\eta}} \ge t
\right)
\le \frac{C_{\ell,p}(N)}{t^{\,p}}.
\]
Equivalently,
\[
\sup_{\eta\ge1}
\mathbb{P}\!\left(
  \mathcal{L}(X^{(\eta)}) \ge t\,\sqrt{\eta}
\right)
\le \frac{C_{\ell,p}(N)}{t^{\,p}}
\qquad (t>0).
\]
In particular, for all \(\epsilon\), 
taking \(t=(C_{\ell,p}(N)/\varepsilon)^{1/p}\) gives
\[\sup_{\eta\ge1}\mathbb{P}\left(\frac{\mathcal{L}(X^{(\eta)})}{\sqrt{\eta}} \ge (C_{\ell,p}(N)/\varepsilon)^{1/p}\right)
\le \varepsilon.\]
\end{corollary}

\begin{proof}
For \(\eta\ge1\),
\[
\mathbb{E}\!\left[
  \left(\frac{\mathcal{L}(X^{(\eta)})}{\sqrt{\eta}}\right)^{\!p}
\right]
=\frac{\mathbb{E}[\mathcal{L}(X^{(\eta)})^{p}]}{\eta^{p/2}}
\le \frac{C^{(0)}_{\ell,p}(N) + C^{(1)}_{\ell,p}(N)\,\eta^{p/2}}{\eta^{p/2}} = C_{\ell,p}(N).
\]
Since \(\mathcal{L}(X^{(\eta)})\ge0\), apply the generalized Markov inequality to 
\(Y_\eta:=\mathcal{L}(X^{(\eta)})/\sqrt{\eta}\):
\[
\mathbb{P}(Y_\eta\ge t)
=\mathbb{P}(Y_\eta^{\,p}\ge t^{\,p})
\le \frac{\mathbb{E}[Y_\eta^{\,p}]}{t^{\,p}}
\le \frac{C_{\ell,p}(N)}{t^{\,p}}.
\]
Taking the supremum over \(\eta\ge1\) proves the claim.
\end{proof}

% ----------------------------------------------------------
\begin{remark}[Role of moment bounds, tightness, and variance scaling]
\label{rem:role_section3}

The results of this section serve a conceptual and methodological role that
goes beyond the specific Gaussian mixture setting in which they are derived.

\medskip
\noindent
\textbf{(i) Beyond trivial global rescaling.}
While a deterministic rescaling of a point cloud trivially rescales all
Vietoris--Rips filtration values, the setting considered in
Theorem~\ref{thm:lp_estimation_gmm} is fundamentally different.
Here, the mean structure of the distribution is fixed, and only the covariance
matrices are scaled.
Thus, the observed $\sqrt{\eta}$--growth of persistence statistics reflects
increasing \emph{stochastic dispersion} around fixed reference locations,
rather than a uniform geometric dilation.
This distinction is essential in applications where scale changes arise from
heterogeneity or volatility, rather than from a deterministic rescaling of the data.

\medskip
\noindent
\textbf{(ii) Tightness and uniform tail control.}
Corollary~\ref{cor:tightness} and
Corollary~\ref{cor:uniform_poly_tail} establish that, after normalization by
$\sqrt{\eta}$, classical persistence statistics form a tight family with
uniform polynomial tail bounds.
These results provide non-asymptotic control over the magnitude and variability
of persistence-based quantities under increasing dispersion.
In particular, they justify the use of scale-normalized or bounded statistics
when comparing random point clouds whose intrinsic variability may differ.

\medskip
\noindent
\textbf{(iii) Motivation for bounded and normalized statistics.}
The variance-scaling behavior quantified in
Theorem~\ref{thm:lp_estimation_gmm} highlights a limitation of raw persistence
summaries: their magnitude is dominated by dispersion effects.
This observation motivates the construction in Section~\ref{sec:pljs_stat} of the PL+JS statistic,
which suppresses scale-driven growth by normalization and regularization, while
retaining sensitivity to changes in topological structure.
In this sense, the results of this section provide the probabilistic foundation for
why a bounded, scale-insensitive comparison of persistence diagrams is desirable
for nonparametric inference and change-point detection.

\medskip
\noindent
Taken together, the moment bounds, tightness, and uniform tail estimates clarify
the stochastic behavior of classical persistence statistics and explain their role
as a theoretical benchmark.
They do not directly enter the definition of the PL+JS statistic, but they
motivate its design and delineate the regime in which purely scale-driven
variability should be factored out.
\end{remark}

% -----------------------------------------------
\section{A Persistence-Based Statistic for Comparing Random Point Clouds}
% -----------------------------------------------

\subsection{Definition of the PL+JS Statistic}
\label{sec:pljs_stat}
We now introduce a scale--bounded, distribution--free statistic for comparing
persistence diagrams.
The motivation for this construction builds on the results of the preceding sections.

In Sections~3.3 and~3.4, we showed that classical persistence statistics,
such as total and maximum persistence,
typically increase with the dispersion of the underlying data distribution,
and in particular scale with the variance multiplier under Gaussian mixture models.
While this behavior is natural from a geometric viewpoint,
it implies that raw persistence quantities are inherently sensitive to global scale
and may obscure more subtle structural changes.

The PL+JS statistic is designed to mitigate this effect by normalizing persistence
landscapes and comparing them through a bounded, information--theoretic divergence
on a fixed support.
In this way, the resulting statistic suppresses purely scale--driven variation
and enables meaningful comparison of topological structure across samples,
making it suitable for nonparametric change--point detection.
No distributional assumptions (e.g., Gaussian mixtures) are required at this stage;
those models were used earlier only to motivate variance--scaling phenomena.

We fix homological degree \(\ell\in\mathbb{N}_0\), and we define a statistic for quantitatively comparing two persistence diagrams \(D, D' \in \Dgm\) for two data windows \(\mathcal{W}_1\) and \(\mathcal{W}_2\).
The construction is based on the persistence landscape representation (see Appendix~\ref{sec:PL}) and the Jensen-Shannon distance (see Appendix~\ref{sec:js_divergence}) between the associated probability density functions.

Fix \(S>0\) and let \(u_S\in\mathcal{P}([0,S])\) denote the uniform density
\begin{equation}
\label{eq:uniform_density}
  u_S(t):=\frac{1}{S}\,1_{[0,S]}(t),
\end{equation}
where
\[
  \mathcal{P}([0,S])
  :=\{p\in L^1_+([0,S]) : \int_0^S p(t)\,dt=1\}
\]
denotes the set of all probability density functions on \([0,S]\), and \(L^1_+([0,S])\) denotes the cone of nonnegative integrable functions on \([0,S]\).

For \(\ell\in\mathbb N_0\), \(M\in\mathbb N\), and \(D \in \Dgm\), let
\(\lambda_{\ell,m}(D)\colon \mathbb{R}\to[0,\infty)\) \((m=1,\ldots,M)\) denote the first \(M\) persistence landscape layers (with the convention that \(\lambda_{\ell,m}(D)\equiv 0\) if \(m\) exceeds the number of bars), and define
\(\Lambda_\ell^{(M)}\colon\Dgm\to L^1_+([0,S])\) by
\begin{equation}
\label{eq:Lambda}
  \Lambda_\ell^{(M)}(D)(t):=\sum_{m=1}^M \lambda_{\ell,m}(D)(t), \qquad t\in[0,S].
\end{equation}
Set
\begin{equation}
\label{eq:A}
  A_\ell(D):=\int_0^S \Lambda_\ell^{(M)}(D)(t)\,dt.
\end{equation}

\begin{assumption}
We assume that every diagram \(D\in\Dgm\) is supported in \([0,S]\times[0,S]\) (equivalently, \(0\le b<d\le S\) for all \((b,d)\in D\)).
\end{assumption}
Under this assumption, for all \(t\in[0,S]\) and \(D\in\Dgm\),
\begin{equation}
\label{eq:upper_and_lower_of_Lambda_and_A}
\begin{aligned}
  &0\le\Lambda_\ell^{(M)}(D)(t)\le M\min\{t,S-t\}\le\frac12 MS, \\
  &0\le A_\ell(D)\le M\!\int_0^S\!\min\{t,S-t\}\,dt=\frac{MS^2}{4}=:A_{\mathrm{max}}.
\end{aligned}
\end{equation}

For \(\theta\in(0,1]\), define the regularized normalization operator \(\mathsf{Reg}_\theta\colon L^1_+([0,S])\to\mathcal{P}([0,S])\) by
\[
  \mathsf{Reg}_\theta(f)(t):=\frac{\,f(t)+\theta A_{\mathrm{max}}\,u_S(t)\,}{\int_0^S f(s)\,ds+\theta A_{\mathrm{max}}}.
\]
For \(\gamma\in[0,1]\) and any \(p\in\mathcal{P}([0,S])\), define the \(\gamma\)-mixture operator
\(\mathsf{Mix}_\gamma\colon\mathcal{P}([0,S])\to\mathcal{P}([0,S])\) by
\[
  \mathsf{Mix}_\gamma(p):=(1-\gamma)\,p+\gamma\,u_S.
\]

\begin{definition}
\label{def:js_static}
Fix \(\ell\in\mathbb N_0\), \(M\in\mathbb{N}\), and \(S>0\).
For \(\gamma\in[0,1],\theta\in(0,1]\), define the PDF associated with \(D\) by
\[
  p_\ell^{(\gamma,\theta)}(D)
  := \mathsf{Mix}_\gamma\big(\mathsf{Reg}_\theta\big(\Lambda_\ell^{(M)}(D)\big)\big)\in\mathcal{P}([0,S]).
\]
Equivalently,
\begin{align*}
  p_\ell^{(\gamma,\theta)}(D)(t)
  &=(1-\gamma)\,\frac{\Lambda_\ell^{(M)}(D)(t)+\theta A_{\mathrm{max}}\,u_S(t)}{A_\ell(D)+\theta A_{\mathrm{max}}}
    +\gamma\,u_S(t) \\
  &=(1-\gamma)\,\frac{\Lambda_\ell^{(M)}(D)(t)}{A_\ell(D)+\theta A_{\mathrm{max}}}
    +\Bigl\{\gamma+(1-\gamma)\tfrac{\theta A_{\mathrm{max}}}{A_\ell(D)+\theta A_{\mathrm{max}}}\Bigr\}u_S(t).
\end{align*}
For two persistence diagrams \(D, D' \in \Dgm\), we define the \emph{Persistence Landscape + Jensen-Shannon} (PL+JS) statistic as
\[
  \Delta_\ell(D,D')
  := d_{\mathrm{JS}}\bigl(p_\ell^{(\gamma,\theta)}(D),\,p_\ell^{(\gamma,\theta)}(D')\bigr),
\]
where \(d_{\mathrm{JS}}\) denotes the Jensen-Shannon distance with logarithm base \(2\) (so that \(\Delta_\ell\in[0,1]\)).
\end{definition}

% ---------------------------------------------------------
% ---------------------------------------------------------
\begin{remark}[Interpretability, boundedness, and parameter control]
\label{rem:js_interpretability_compact}
The use of the Jensen--Shannon distance in Definition~\ref{def:js_static}
yields a statistic \(\Delta_\ell\) that is bounded in \([0,1]\),
providing a scale-free and directly interpretable measure of discrepancy
between persistence diagrams.
This boundedness is particularly advantageous for finite-sample nonparametric inference, including permutation-based testing, as it prevents a small number of extreme persistence features from dominating the statistic.

The regularization and mixing parameters \(\theta\) and \(\gamma\)
stabilize the normalization of persistence landscapes when the total
landscape mass is small.
Their influence is explicit and quantitative: as shown in
Proposition~\ref{prop:Continuity_Delta} and
Theorem~\ref{thm:continuity_T}, increasing either parameter improves
robustness by reducing the H\"older constant, while smaller values
enhance sensitivity to topological signal.
Thus, \((\gamma,\theta)\) provide principled control of the
stability--sensitivity trade-off rather than ad hoc regularization.
\end{remark}

% ---------------------------------------------------------
\begin{remark}[Landscape sums versus silhouettes]
\label{rem:landscape_vs_silhouette_compact}
Persistence silhouettes aggregate all triangular functions
associated with a diagram into a single weighted average,
emphasizing long-lived features through nonlinear persistence weights
\cite{chazal2015}.
In contrast, the present work employs the unweighted sum of the first
\(M\) persistence landscape layers,
\(\Lambda_\ell^{(M)}=\sum_{m=1}^M \lambda_{\ell,m}\).

This choice is motivated by transparency, stability, and practical
controllability.
The parameter \(M\) governs the amount of topological information
included in the statistic in a monotone and easily interpretable manner,
and the resulting functional inherits a Lipschitz-type stability bound
that grows linearly in \(M\).

From a practical perspective, \(M\) can be selected in an adaptive
manner.
One may begin with a small value of \(M\), capturing only the most
dominant topological features, and assess whether structural changes
are detected.
If no significant differences are observed, \(M\) can be increased
to incorporate finer-scale features.
Since increasing \(M\) leads to higher computational cost, its choice
naturally reflects a trade-off between statistical sensitivity and
computational efficiency.

While silhouettes provide an effective low-dimensional summary,
the landscape-sum representation retains an explicit hierarchical
ordering of dominant features.
This hierarchical structure, together with the tunable parameter \(M\),
is particularly advantageous for detecting structural changes in a
statistically controlled and computationally feasible manner.
\end{remark}

% ---------------------------------------------------------
\begin{remark}[Global cutoff and comparability]
\label{rem:global_S}
Fixing the cutoff \(S\) globally ensures that all probability densities
associated with persistence diagrams share the same support \([0,S]\).
As a result, Jensen--Shannon distances computed across different pairs
of diagrams are strictly comparable.
In practice, \(S\) should be chosen slightly larger than the maximal
finite death time expected in the data to avoid truncation effects.
\end{remark}

% ---------------------------------------------------------
\begin{remark}[Limits of signature-based testing]
\label{rem:information_loss_compact}
Any test based on a compressed signature of persistence diagrams
necessarily identifies distributions only up to the equivalence
relation induced by that signature.
Distinct persistence diagrams may therefore become indistinguishable
after compression; this phenomenon is intrinsic to any signature-based
representation and is well illustrated, for example, by the notion of
\emph{Euler equivalence} discussed by D{\l}otko et al.~\cite{dlotko2024}
for the Euler characteristic curve.

Accordingly, the PL+JS statistic is not intended as a complete invariant
of persistence diagrams, but as a stable and inference-oriented summary.
It is designed to provide reliable power against a broad class of
structured alternatives while sacrificing sensitivity to fine-scale
differences that are often dominated by noise or sampling variability.
\end{remark}

% --------------------------------------------------
\subsection{Continuity of the Proposed Change-Point Detection Statistic}
\label{sec:math_basis_change_point_detection}

In what follows, we reuse the notation introduced in
Section~\ref{sec:general_preliminaries}.
Let \(\mathcal{X}=\mathbb{R}^d\) with \(d\in\mathbb{N}\), where each element
of \(\mathcal{X}\) represents an individual attribute vector.
For each \(N\in\mathbb{N}\), let
\[
\mathcal{D}_\ell^{(N)} \colon \mathcal{X}^N \to \Dgm
\]
denote the map that assigns to a finite point cloud
\(x=(x_1,\dots,x_N)\) the persistence diagram in degree
\(\ell\in\mathbb{N}_0\) associated with the Vietoris--Rips filtration.
When no confusion arises, we simply write \(\mathcal{D}_\ell\) instead of
\(\mathcal{D}_\ell^{(N)}\).

% In what follows, we reuse the notation introduced in Section~\ref{sec:general_preliminaries}.
% Let \(\mathcal{X}\) be \(\mathbb{R}^d\) with \(d\in\mathbb{N}\), where each element of \(\mathcal{X}\) represents an individual attribute vector.
% We defined the map \(\mathcal{D}_\ell=\mathcal{D}_\ell^{(N)} \colon \mathcal{X}^N \to \Dgm\) that assigns, to a finite point cloud \(x = (x_1, \dots, x_N)\), the persistence diagram in degree \(\ell \in \mathbb{N}_0\) associated with the Vietoris-Rips filtration. 
% Here, since the numbers of points \(N_1\) and \(N_2\) may differ for two data windows \(\mathcal{W}_1\) and \(\mathcal{W}_2\), we explicitly indicate the dependence on the number of points by writing \(\mathcal{D}_\ell^{(N_1)}\) and \(\mathcal{D}_\ell^{(N_2)}\).

Throughout, we fix a homological
degree \(\ell\in\mathbb{N}_0\), a positive integer \(M\), a global cutoff \(S>0\) and \(\gamma\in[0,1],\theta\in(0,1]\) as in
Definition~\ref{def:js_static}.

To verify the stability of the proposed PL+JS statistic with respect to variations in the observed data, we prove the continuity of the mapping 
\(T_\ell\) below. 

\begin{theorem}
\label{thm:continuity_T}
Let \(T_\ell \colon \mathcal{X}^{N_1} \times \mathcal{X}^{N_2} \to \mathbb{R}\) be the map
\[
T_\ell(x^{(1)},x^{(2)})=\Delta_\ell(\mathcal{D}_\ell(x^{(1)}), \mathcal{D}_\ell(x^{(2)})),
\]
where \(\Delta_\ell \colon \Dgm \times \Dgm \to \mathbb{R}\) is the map defined in Definition~\ref{def:js_static}. 
Equip 
% \(\mathcal{X}^{N_i}\) with the metric 
% \[
% d_{\mathcal{X}^{N_i}}(x^{(i)},\tilde x^{(i)}):=\max_{1\le j\le N_i}\|x_j^{(i)}-\tilde x_j^{(i)}\|,
% \]
% and 
\(\mathcal{X}^{N_1}\times\mathcal{X}^{N_2}\) with the metric 
\[
d_{\mathcal{X}^{N_1}\times\mathcal{X}^{N_2}}\left((x^{(1)},x^{(2)}),(\tilde x^{(1)},\tilde x^{(2)})\right)
% :=d_{\mathcal{X}^{N_1}}(x^{(1)},\tilde x^{(1)})+d_{\mathcal{X}^{N_2}}(x^{(2)},\tilde x^{(2)}).
:=d_H^{(N_1)}(x^{(1)},\tilde x^{(1)})
 +d_H^{(N_2)}(x^{(2)},\tilde x^{(2)}),
\]
where 
\[
d_H^{(N)}(x,\tilde x)
:=\max\Bigl\{
\max_{1\le i\le N}\min_{1\le j\le N}\|x_i-\tilde x_j\|,
\ \max_{1\le j\le N}\min_{1\le i\le N}\|x_i-\tilde x_j\|
\Bigr\}
\]
defines a Hausdorff-type metric on \(\mathcal{X}^N\).
Then \(T_\ell\) is \(1/2\)-H\"{o}lder continuous with respect to this metric:
\[
\left|T_\ell(x^{(1)},x^{(2)})-T_\ell(\tilde x^{(1)},\tilde x^{(2)})\right|
\le C_T
\sqrt{d_{\mathcal{X}^{N_1}\times\mathcal{X}^{N_2}}\left((x^{(1)},x^{(2)}),(\tilde x^{(1)},\tilde x^{(2)})\right)}
\]
where
\[
C_T:=\sqrt{2SC_{\phi}C_{\mathrm{Density}}},
\]
\[
C_{\phi}:=
\frac{1+\theta}{(\gamma+\theta)\log 2}\left(1+\frac{2(1-\gamma)}{\theta}\right)
,\quad
C_{\mathrm{Density}}
:= \frac{8(1-\gamma)}{S^2}
\left(\frac{1}{\theta}+\frac{1}{\theta^2}\right).
\]
\end{theorem}

% \begin{remark}
% Here and in what follows, we identify a point cloud
% \(x=(x_1,\dots,x_N)\in\mathcal{X}^N\) with the corresponding finite subset
% \(\{x_1,\dots,x_N\}\subset\mathcal{X}\).
% Possible multiplicities of identical points are ignored, since both the
% Vietoris--Rips complex and the Hausdorff distance between finite subsets
% of \(\mathcal{X}\) are invariant under repetition of points.
% \end{remark}

As a preliminary step toward proving Theorem~\ref{thm:continuity_T},  
we first verify the continuity of the map \(\Delta_\ell \colon \Dgm \times \Dgm \to \mathbb{R}\), which appears as a component of \(T_\ell\).

% -----------------------------------------------
\begin{proposition}
\label{prop:Continuity_Delta}
Let 
\[
d_\times\bigl((D,D'),(\tilde D, \tilde D')\bigr)
:=d_B(D,\tilde D)+d_B(D',\tilde D')
\]
be the metric on \(\Dgm\times\Dgm\)
where \(d_B\) denotes the bottleneck distance (Appendix~\ref{sec:bottleneck_distance}). 
The map
\[
  \Delta_\ell \colon \Dgm \times \Dgm \to \mathbb{R},
  \qquad
  (D, D')\longmapsto
  d_{\mathrm{JS}}\bigl(p_\ell^{(\gamma,\theta)}(D),p_\ell^{(\gamma,\theta)}(D')\bigr)
\]
is \(1/2\)-H\"{o}lder continuous with respect to \(d_\times\). 
More precisely, for all \((D,D'),(\tilde D,\tilde D')\in\Dgm\times\Dgm\), we have
\begin{align*}
\left|\Delta_\ell(D,D')-\Delta_\ell(\tilde{D},\tilde{D}')\right|
&\le
C_{\Delta}
\sqrt{d_\times\bigl((D,D'),(\tilde D, \tilde D')\bigr)},
\end{align*}
where
\[
C_{\Delta}:=\sqrt{SC_{\phi}C_{\mathrm{Density}}},
\]
\[
C_{\phi}
=\frac{1+\theta}{(\gamma+\theta)\log 2}\left(1+\frac{2(1-\gamma)}{\theta}\right)
, \quad
C_{\mathrm{Density}}=\frac{8(1-\gamma)}{S^2}\Bigl(\frac{1}{\theta}+\frac{1}{\theta^2}\Bigr).
\]
\end{proposition}

We split the proof into several lemmas. 

% -----------------------------------------------
\begin{lemma}[Continuity of the Sum of Persistence Landscapes]
\label{lem:continuity_sum_of_landscapes}
The map 
\[
  \Dgm\times[0,S]\to\mathbb{R}, \quad
  (D,t)\mapsto \Lambda_\ell^{(M)}(D)(t)
\]
is continuous, where \(\Lambda_\ell^{(M)}\) is defined in Definition~\ref{eq:upper_and_lower_of_Lambda_and_A}.
Moreover, for all \(D,D'\in\Dgm\),
\begin{equation}
\label{eq:stability_bottleneck}
  \Vert\Lambda_\ell^{(M)}(D)-\Lambda_\ell^{(M)}(D')\Vert_\infty
  \le M\,d_B(D,D').
\end{equation}
\end{lemma}

\begin{proof}
The Lipschitz continuity of each persistence landscape layer 
with respect to the bottleneck distance is a standard stability result 
(see, e.g., Bubenik~\cite{bubenik2015}). 
In particular, for each \(m\in\{1,\dots,M\}\),
\[
  \|\lambda_{\ell,m}(D)-\lambda_{\ell,m}(D')\|_\infty 
  \le d_B(D,D').
\]
Hence,
\[
  \Vert\Lambda_\ell^{(M)}(D)-\Lambda_\ell^{(M)}(D')\Vert_\infty
  \le \sum_{m=1}^M
  \Vert\lambda_{\ell,m}(D)-\lambda_{\ell,m}(D')\Vert_\infty 
  \le M\,d_B(D,D').
\]
Therefore, the map \(\Dgm \to C([0,S])\), \(D\mapsto\Lambda_\ell^{(M)}(D)\), is continuous.
Since the evaluation map 
\(C([0,S])\times[0,S]\to\mathbb{R}, (f,t)\mapsto f(t)\),
is continuous, the composition 
\((D,t)\mapsto\Lambda_\ell^{(M)}(D)(t)\)
is continuous.
\end{proof}

% -----------------------------------------------
\begin{lemma}[Lipschitz Continuity and Uniform Bounds for the Density]
\label{lem:continuity_of_density}
The density \(p_\ell^{(\gamma,\theta)}\), defined in Definition~\ref{def:js_static}, satisfies
\begin{equation}
\label{eq:upper_lower_of_density}
  a_* \le p_\ell^{(\gamma,\theta)}(D)(t) \le a^*,
  \qquad t\in[0,S],\ D\in\Dgm,
\end{equation}
where, in particular, one may take
\[
  a_*:=\frac{\gamma+\theta}{(1+\theta)S},\qquad
  a^*:=\frac{1}{S}\Bigl(1+\frac{2(1-\gamma)}{\theta}\Bigr).
\]
Moreover, the map \(p_\ell^{(\gamma,\theta)}\colon(\Dgm,d_B)\to (C([0,S]),\|\cdot\|_\infty)\), \(D\mapsto p_\ell^{(\gamma,\theta)}(D)\), is Lipschitz continuous. More precisely, for all \(D,D'\in\Dgm\),
\begin{equation}
\label{eq:density_lipschitz}
  \|p_\ell^{(\gamma,\theta)}(D)-p_\ell^{(\gamma,\theta)}(D')\|_\infty
  \le C_{\mathrm{Density}}\, d_B(D,D'),
\end{equation}
where one may take
\[
  C_{\mathrm{Density}}:=\frac{8(1-\gamma)}{S^2}\Bigl(\frac{1}{\theta}+\frac{1}{\theta^2}\Bigr).
\]
\end{lemma}

\begin{proof}
Note that, from \eqref{eq:upper_and_lower_of_Lambda_and_A},
\begin{equation}
\label{eq:upper_by_A_max}
  \|\Lambda_\ell^{(M)}(D)\|_\infty \le \frac{1}{2}MS = \frac{2A_{\mathrm{max}}}{S}
  \qquad\text{for all }D\in\Dgm.
\end{equation}
By Definition~\ref{def:js_static},
\[
  p_\ell^{(\gamma,\theta)}(D)(t)
  =(1-\gamma)\frac{\Lambda_\ell^{(M)}(D)(t)}{A_\ell(D)+\theta A_{\mathrm{max}}}
  +\Bigl\{\gamma+(1-\gamma)\tfrac{\theta A_{\mathrm{max}}}{A_\ell(D)+\theta A_{\mathrm{max}}}\Bigr\}u_S(t).
\]
Since \(A_\ell(D)\le A_{\mathrm{max}}\), the coefficient of \(u_S(t)\) is bounded below by
\(\gamma+(1-\gamma)\frac{\theta}{1+\theta}\). Hence
\[
  p_\ell^{(\gamma,\theta)}(D)(t)
  \ge \Bigl(\gamma+(1-\gamma)\tfrac{\theta}{1+\theta}\Bigr)\,u_S(t)
  = \frac{\gamma+\theta}{(1+\theta)S}.
\]
For the upper bound, using \eqref{eq:upper_by_A_max} and \(A_\ell(D)+\theta A_{\mathrm{max}}\ge \theta A_{\mathrm{max}}\),
\[
  p_\ell^{(\gamma,\theta)}(D)(t)
  \le (1-\gamma)\frac{\frac{2A_{\mathrm{max}}}{S}}{\theta A_{\mathrm{max}}} + u_S(t)
  = \frac{1}{S}\Bigl(\frac{2(1-\gamma)}{\theta}+1\Bigr).
\]

Next, write
\begin{align*}
&p_\ell^{(\gamma,\theta)}(D)(t)-p_\ell^{(\gamma,\theta)}(D')(t) \\
&=(1-\gamma)\!\left(\frac{\Lambda_\ell^{(M)}(D)(t)}{A_\ell(D)+\theta A_{\mathrm{max}}}
-\frac{\Lambda_\ell^{(M)}(D')(t)}{A_\ell(D')+\theta A_{\mathrm{max}}}\right) \\
&+(1-\gamma)\theta A_{\mathrm{max}}\!\left(
\frac{1}{A_\ell(D)+\theta A_{\mathrm{max}}}
-\frac{1}{A_\ell(D')+\theta A_{\mathrm{max}}}\right)u_S(t).
\end{align*}
Hence
\begin{equation}
\label{eq:density_estimation_in_the_proof}
  \|p_\ell^{(\gamma,\theta)}(D)-p_\ell^{(\gamma,\theta)}(D')\|_\infty
  \le (1-\gamma)\Bigl(I_1+\frac{\theta A_{\mathrm{max}}}{S}I_2\Bigr),
\end{equation}
where
\[
  I_1:=\left\|\frac{\Lambda_\ell^{(M)}(D)}{A_\ell(D)+\theta A_{\mathrm{max}}}
           -\frac{\Lambda_\ell^{(M)}(D')}{A_\ell(D')+\theta A_{\mathrm{max}}}\right\|_\infty,
\]
\[
  I_2:=\left| \frac{1}{A_\ell(D)+\theta A_{\mathrm{max}}}
             -\frac{1}{A_\ell(D')+\theta A_{\mathrm{max}}}\right|.
\]
From \eqref{eq:upper_by_A_max} and the triangle inequality, we have
\begin{align*}
I_1&\le
\left\|\frac{\Lambda_\ell^{(M)}(D)-\Lambda_\ell^{(M)}(D')}{A_\ell(D)+\theta A_{\mathrm{max}}}\right\|_\infty+
\left|\frac{1}{A_\ell(D)+\theta A_{\mathrm{max}}}-\frac{1}{A_\ell(D')+\theta A_{\mathrm{max}}}\right|
\left\|\Lambda_\ell^{(M)}(D')\right\|_\infty \\
&=
\frac{\left\|\Lambda_\ell^{(M)}(D)-\Lambda_\ell^{(M)}(D')\right\|_\infty}{A_\ell(D)+\theta A_{\mathrm{max}}}+
\frac{|A_\ell(D)-A_\ell(D')|}{(A_\ell(D)+\theta A_{\mathrm{max}})(A_\ell(D')+\theta A_{\mathrm{max}})}
\frac{2A_{\mathrm{max}}}{S} \\
&\le
\frac{\left\|\Lambda_\ell^{(M)}(D)-\Lambda_\ell^{(M)}(D')\right\|_\infty}{\theta A_{\mathrm{max}}}+
\frac{2|A_\ell(D)-A_\ell(D')|}{\theta^2 A_{\mathrm{max}}S}.
\end{align*}

Moreover,
\[
|A_\ell(D)-A_\ell(D')|
\le \int_0^S |\Lambda_\ell^{(M)}(D)(t)-\Lambda_\ell^{(M)}(D')(t)|\,dt
\le S\,\|\Lambda_\ell^{(M)}(D)-\Lambda_\ell^{(M)}(D')\|_\infty,
\]
whence
\begin{align*}
I_1&\le\frac{1}{A_{\mathrm{max}}}
\left(\frac{1}{\theta}+
\frac{2}{\theta^2}\right)
\left\|\Lambda_\ell^{(M)}(D)-\Lambda_\ell^{(M)}(D')\right\|_\infty, \\
I_2&\le\frac{|A_\ell(D)-A_\ell(D')|}{\theta^2 A_{\mathrm{max}}^2}
\le
\frac{S}{\theta^2 A_{\mathrm{max}}^2}\left\|\Lambda_\ell^{(M)}(D)-\Lambda_\ell^{(M)}(D')\right\|_\infty.
\end{align*}
Plugging these into \eqref{eq:density_estimation_in_the_proof} yields
\begin{equation*}
\|p_\ell^{(\gamma,\theta)}(D)-
p_\ell^{(\gamma,\theta)}(D')\|_\infty\le
\frac{2(1-\gamma)}{A_{\mathrm{max}}}
\left(\frac{1}{\theta}+\frac{1}{\theta^2}\right)
\left\|\Lambda_\ell^{(M)}(D)-\Lambda_\ell^{(M)}(D')\right\|_\infty.
\end{equation*}
Finally, by \(\|\Lambda_\ell^{(M)}(D)-\Lambda_\ell^{(M)}(D')\|_\infty \le M\,d_B(D,D')\)
(see Lemma~\ref{lem:continuity_sum_of_landscapes}) and \(M=\frac{4A_{\mathrm{max}}}{S^2}\), we obtain \eqref{eq:density_lipschitz}.
\end{proof}

% -----------------------------------------------
\begin{lemma}[Continuity of the JS Divergence with respect to Persistence Diagrams]
\label{lem5:continuity_Delta}
Let 
\[
m(D,D')(t):=\tfrac12\bigl(p_\ell^{(\gamma,\theta)}(D)(t)+p_\ell^{(\gamma,\theta)}(D')(t)\bigr).
\] 
The JS divergence with logarithm base \(2\) can then expressed as
\begin{align*}
  \mathrm{JS}(p_\ell^{(\gamma,\theta)}(D)\Vert p_\ell^{(\gamma,\theta)}(D'))
  & := \frac12\int_0^S \phi\bigl(p_\ell^{(\gamma,\theta)}(D)(t),m(D,D')(t)\bigr)\,dt \\
  & \phantom{:}+ \frac12\int_0^S \phi\bigl(p_\ell^{(\gamma,\theta)}(D')(t),m(D,D')(t)\bigr)\,dt,
\end{align*}
where 
\[
\phi(u,v):=
u\log_2\frac{u}{v}
=\frac{1}{\log2}u\log\frac{u}{v}
, \qquad (u,v)\in(0,\infty)^2.
\]
(Throughout, \(log\) denotes the natural logarithm, and \(log_2\) denotes the base-2 logarithm.)
Let 
\[
d_\times\bigl((D,D'),(\tilde D, \tilde D')\bigr)
:=d_B(D,\tilde D)+d_B(D',\tilde D')
\]
be the metric on \(\Dgm\times\Dgm\).

\smallskip
\noindent Then the map \(\Dgm\times\Dgm\to\mathbb{R}\) defined by 
\((D,D')\longmapsto \mathrm{JS}\bigl(p_\ell^{(\gamma,\theta)}(D)\Vert p_\ell^{(\gamma,\theta)}(D')\bigr)\)
is globally Lipschitz continuous with respect to \(d_\times\):
\[
\left|\mathrm{JS}\bigl(p_\ell^{(\gamma,\theta)}(D)\Vert p_\ell^{(\gamma,\theta)}(D')\bigr)-
\mathrm{JS}\bigl(p_\ell^{(\gamma,\theta)}(\tilde{D})\Vert p_\ell^{(\gamma,\theta)}(\tilde{D}')\bigr)\right|
\le C_{\mathrm{JS}}
\, d_\times\!\bigl((D,D'),(\tilde D,\tilde D')\bigr),
\]
where 
\[
C_{\mathrm{JS}}:=SC_{\phi}C_{\mathrm{Density}},
\]
\[
C_{\phi}
=\frac{1+\theta}{(\gamma+\theta)\log 2}\left(1+\frac{2(1-\gamma)}{\theta}\right)
, \
C_{\mathrm{Density}}=\frac{8(1-\gamma)}{S^2}\Bigl(\frac{1}{\theta}+\frac{1}{\theta^2}\Bigr).
\]
\end{lemma}

\begin{proof}
From Lemma~\ref{lem:continuity_of_density}, for all \(t\in[0,S]\) and \(D\in\Dgm\),
\[
a_*\le p_\ell^{(\gamma,\theta)}(D)(t)\le a^*,
\]
hence also 
\[
a_*\le m(D,D')(t)\le a^*
\]
for all \(D,D'\in\Dgm\). 

We show that \(\phi\) is Lipschitz continuous on the compact rectangle \([a_*,a^*]^2\) with respect to the \(1\)-norm \(\Vert(u,v)\Vert_1:=|u|+|v|\).
Since 
\[
\frac{\partial\phi}{\partial u}(u,v)=\frac1{\log2}\left(\log\frac{u}{v}+1\right),\quad
\frac{\partial\phi}{\partial v}(u,v)=-\frac{u}{v\,\log2},
\]
we have for all \(u,v\in[a_*,a^*]\), 
\begin{align*}
\left|\frac{\partial\phi}{\partial u}(u,v)\right|&\le
\frac{1}{\log2}\left[\max\left\{
\left|\log\frac{a^*}{a_*}\right|, 
\left|\log\frac{a_*}{a^*}\right|\right\}+1\right] \\
&=\frac{1}{\log2}\left(\log\frac{a^*}{a_*}+1\right)
\le\frac{a^*}{a_*\,\log2}
=\frac{1+\theta}{(\gamma+\theta)\log 2}\left\{1+\frac{2(1-\gamma)}{\theta}\right\}=C_{\phi},\\
\left|\frac{\partial\phi}{\partial v}(u,v)\right|&\le
\frac{a^*}{a_*\,\log2}=C_{\phi}.
\end{align*}
 
Therefore we have for all \((u,v),(u',v')\in[a_*,a^*]^2\),
\begin{align}
\notag
|\phi(u,v)-\phi(u',v')|
&=\left|\int_0^1\frac{\partial}{\partial t}\phi(u'+t(u-u'),v'+t(v-v'))dt\right| \\
\notag
&\le
\int_0^1\left|\frac{\partial\phi}{\partial u}(u'+t(u-u'),v'+t(v-v'))(u-u')\right|dt \\
\notag
&+
\int_0^1\left|\frac{\partial\phi}{\partial v}(u'+t(u-u'),v'+t(v-v'))(v-v')\right|dt \\
\label{eq:phi_Lipschitz}
&\le
C_{\phi}\left(|u-u'|+|v-v'|\right).
\end{align}

From the inequality \eqref{eq:density_lipschitz}, we have for all \((D,D'),(\tilde{D},\tilde{D}')\in\Dgm\times\Dgm\),
\begin{align*}
\Vert m(D,D')-m(\tilde{D},\tilde{D}')\Vert_\infty
&\le\frac12\Vert p_\ell^{(\gamma,\theta)}(D)-p_\ell^{(\gamma,\theta)}(\tilde{D})\Vert_\infty
+\frac12\Vert p_\ell^{(\gamma,\theta)}(D')-p_\ell^{(\gamma,\theta)}(\tilde{D}')\Vert_\infty \\
&\le\frac12C_{\mathrm{Density}}\left(d_B(D,\tilde{D})+d_B(D',\tilde{D}')\right).
\end{align*}
Using the inequality \eqref{eq:phi_Lipschitz}, we get 
\begin{align*}
&\left|\mathrm{JS}(p_\ell^{(\gamma,\theta)}(D)\Vert p_\ell^{(\gamma,\theta)}(D'))
-\mathrm{JS}(p_\ell^{(\gamma,\theta)}(\tilde{D})\Vert p_\ell^{(\gamma,\theta)}(\tilde{D}'))\right| \\
& \le \frac12\int_0^S \left|\phi\bigl(p_\ell^{(\gamma,\theta)}(D)(t),m(D,D')(t)\bigr)
-\phi\bigl(p_\ell^{(\gamma,\theta)}(\tilde{D})(t),m(\tilde{D},\tilde{D}')(t)\bigr)\right|\,dt \\
& + \frac12\int_0^S \left|\phi\bigl(p_\ell^{(\gamma,\theta)}(D')(t),m(D,D')(t)\bigr)
-\phi\bigl(p_\ell^{(\gamma,\theta)}(\tilde{D}')(t),m(\tilde{D},\tilde{D}')(t)\bigr)\right|\,dt \\
& \le \frac12C_{\phi}\int_0^S \left\{\left|p_\ell^{(\gamma,\theta)}(D)(t)-p_\ell^{(\gamma,\theta)}(\tilde{D})(t)\right|
+\left|m(D,D')(t)-m(\tilde{D},\tilde{D}')(t)\right|\right\}\,dt \\
& + \frac12C_{\phi}\int_0^S \left\{\left|p_\ell^{(\gamma,\theta)}(D')(t)-p_\ell^{(\gamma,\theta)}(\tilde{D}')(t)\right|
+\left|m(D,D')(t)-m(\tilde{D},\tilde{D}')(t)\right|\right\}\,dt \\
& \le \frac12SC_{\phi}\left(\left\Vert p_\ell^{(\gamma,\theta)}(D)-p_\ell^{(\gamma,\theta)}(\tilde{D})\right\Vert_\infty
+\left\Vert m(D,D')-m(\tilde{D},\tilde{D}')\right\Vert_\infty\right) \\
& + \frac12SC_{\phi}\left(\left\Vert p_\ell^{(\gamma,\theta)}(D')-p_\ell^{(\gamma,\theta)}(\tilde{D}')\right\Vert_\infty
+\left\Vert m(D,D')-m(\tilde{D},\tilde{D}')\right\Vert_\infty\right) \\
& \le \frac12SC_{\phi}C_{\mathrm{Density}}\left\{d_B(D,\tilde{D})
+\frac12\left(d_B(D,\tilde{D})+d_B(D',\tilde{D}')\right)\right\} \\
& + \frac12SC_{\phi}C_{\mathrm{Density}}\left\{d_B(D',\tilde{D}')
+\frac12\left(d_B(D,\tilde{D})+d_B(D',\tilde{D}')\right)\right\} \\
& = SC_{\phi}C_{\mathrm{Density}}\left(d_B(D,\tilde{D})+d_B(D',\tilde{D}')\right).
\end{align*}
This shows global Lipschitz continuity (hence continuity).
\end{proof}

% -----------------------------------------------
\begin{proof}[Proof of Proposition~\ref{prop:Continuity_Delta}]

Since we have 
\[
|\sqrt{a}-\sqrt{b}|\le\sqrt{|a-b|}
\]
for all \(a,b\ge0\), applying Lemma~\ref{lem5:continuity_Delta}, we have for all \((D,D'),(\tilde{D},\tilde{D}')\in \Dgm\times\Dgm\),
\begin{align*}
\left|\Delta_\ell(D,D')-\Delta_\ell(\tilde{D},\tilde{D}')\right|
&=
\left|d_{\mathrm{JS}}(p_\ell^{(\gamma,\theta)}(D), p_\ell^{(\gamma,\theta)}(D'))-d_{\mathrm{JS}}(p_\ell^{(\gamma,\theta)}(\tilde{D}), p_\ell^{(\gamma,\theta)}(\tilde{D}'))\right|\\
&=
\left|\sqrt{\mathrm{JS}(p_\ell^{(\gamma,\theta)}(D)\Vert p_\ell^{(\gamma,\theta)}(D'))}-\sqrt{\mathrm{JS}(p_\ell^{(\gamma,\theta)}(\tilde{D})\Vert p_\ell^{(\gamma,\theta)}(\tilde{D}'))}\right|\\
&\le
\sqrt{\left|\mathrm{JS}(p_\ell^{(\gamma,\theta)}(D),p_\ell^{(\gamma,\theta)}(D'))-\mathrm{JS}(p_\ell^{(\gamma,\theta)}(\tilde{D})\Vert p_\ell^{(\gamma,\theta)}(\tilde{D}'))\right|}\\
&\le
\sqrt{C_{\mathrm{JS}}d_\times((D,D'),(\tilde D,\tilde D'))}
=C_\Delta\sqrt{d_\times((D,D'),(\tilde D,\tilde D'))}.
\end{align*}
This proves the \(1/2\)-H\"{o}lder continuity of \(\Delta_\ell\).
\end{proof}

% ------------------------------------------
Before proving Theorem~\ref{thm:continuity_T}, we prepare a standard stability
result for Vietoris--Rips persistence diagrams with respect to the Hausdorff
distance. This result will serve as a key technical ingredient in the proof.
\begin{proposition}[Stability of Vietoris--Rips persistence diagrams]
\label{prop:vr_stability}
Let \(X,Y\subset\mathbb R^d\) be nonempty finite sets, and let
\(\mathrm{VR}(X)=\{\mathrm{VR}(X,r)\}_{r\ge0}\) and
\(\mathrm{VR}(Y)=\{\mathrm{VR}(Y,r)\}_{r\ge0}\) denote the Vietoris--Rips
filtrations parametrized by the diameter, i.e.,
\[
\mathrm{VR}(X,r)
:=\{\sigma\subset X \mid \max_{x,x'\in\sigma}\|x-x'\|\le r\}.
\]
Then, for each homological degree \(\ell\ge0\), the corresponding persistence
diagrams satisfy
\[
d_B\bigl(D_\ell(\mathrm{VR}(X)),D_\ell(\mathrm{VR}(Y))\bigr)
\;\le\;
2\,d_H(X,Y),
\]
where \(d_B\) denotes the bottleneck distance and \(d_H\) the Hausdorff distance.
\end{proposition}

\begin{proof}
First, set \(\varepsilon:=d_H(X,Y)\).

By the definition of the Hausdorff distance, for any \(x\in X\) there exists
a point \(y\in Y\) such that \(\|x-y\|\le\varepsilon\).
For each \(x\in X\), choose one such point and define \(f(x)\in Y\).
Similarly, for each \(y\in Y\), choose a point \(x\in X\) satisfying
\(\|y-x\|\le\varepsilon\) and define \(g(y)\in X\).
Then it follows that
\(\|x-f(x)\|\le\varepsilon\) for all \(x\in X\), and
\(\|y-g(y)\|\le\varepsilon\) for all \(y\in Y\).

Fix an arbitrary \(r\ge0\).

Take an arbitrary simplex \(\{x_0,\dots,x_k\}\subset X\) contained in
\(\mathrm{VR}(X,r)\).
For any \(i,j\), by the definition of \(\mathrm{VR}(X,r)\) we have
\(\|x_i-x_j\|\le r\), and hence by the triangle inequality,
\[
\|f(x_i)-f(x_j)\|
\le \|f(x_i)-x_i\|+\|x_i-x_j\|+\|x_j-f(x_j)\|
\le r+2\varepsilon.
\]
Therefore, \(\{f(x_0),\ldots,f(x_k)\}\) is a simplex in
\(\mathrm{VR}(Y,r+2\varepsilon)\).
This correspondence defines a \emph{simplicial map}
\[
F_r:\mathrm{VR}(X,r)\longrightarrow \mathrm{VR}(Y,r+2\varepsilon),
\]
and similarly a simplicial map
\[
G_r:\mathrm{VR}(Y,r)\longrightarrow \mathrm{VR}(X,r+2\varepsilon).
\]

Let \(r'\ge r\) be arbitrary.
Note that for each simplex \(\sigma=\{x_0,\ldots,x_k\}\in\mathrm{VR}(X,r)\),
we define
\[
F_r(\sigma):=\{f(x_0),\ldots,f(x_k)\}.
\]
Since this definition does not depend on \(r\), for
\(\sigma\in\mathrm{VR}(X,r')\) we also have
\[
F_{r'}(\sigma)=\{f(x_0),\ldots,f(x_k)\}.
\]

Consequently,
\[
F_{r'}\bigl(\iota_{r',r}^{(X)}(\sigma)\bigr)
=F_{r'}(\sigma)
=\{f(x_0),\ldots,f(x_k)\}.
\]
On the other hand,
\[
\iota_{r'+2\varepsilon,r+2\varepsilon}^{(Y)}\bigl(F_r(\sigma)\bigr)
=\iota_{r'+2\varepsilon,r+2\varepsilon}^{(Y)}
\bigl(\{f(x_0),\ldots,f(x_k)\}\bigr)
=\{f(x_0),\ldots,f(x_k)\}.
\]
Hence, the following diagram commutes for the family \(\{F_r\}_{r\ge0}\):
\[
\begin{tikzcd}[column sep=huge]
\mathrm{VR}(X,r) \arrow[r,"\iota_{r',r}^{(X)}"] \arrow[d,"F_r"'] 
& \mathrm{VR}(X,r') \arrow[d,"F_{r'}"] \\
\mathrm{VR}(Y,r+2\varepsilon) 
  \arrow[r,"\iota_{r'+2\varepsilon,r+2\varepsilon}^{(Y)}"] 
& \mathrm{VR}(Y,r'+2\varepsilon).
\end{tikzcd}
\]
By the same argument, since the definition of \(G_r\) does not depend on the
parameter \(r\), the family \(\{G_r\}_{r\ge0}\) also satisfies the corresponding
commutativity condition.
More precisely, for any \(r'\ge r\), the diagram
\[
\begin{tikzcd}[column sep=huge]
\mathrm{VR}(Y,r) \arrow[r,"\iota_{r',r}^{(Y)}"] \arrow[d,"G_r"'] 
& \mathrm{VR}(Y,r') \arrow[d,"G_{r'}"] \\
\mathrm{VR}(X,r+2\varepsilon) 
  \arrow[r,"\iota_{r'+2\varepsilon,r+2\varepsilon}^{(X)}"] 
& \mathrm{VR}(X,r'+2\varepsilon)
\end{tikzcd}
\]
commutes.

Next, for any simplex \(\sigma=\{x_0,\ldots,x_k\}\in\mathrm{VR}(X,r)\),
the equality
\[
\iota_{r+4\varepsilon,r}^{(X)}(\sigma)\cup(G_{r+2\varepsilon}\circ F_r)(\sigma)
=\sigma\cup(G_{r+2\varepsilon}\circ F_r)(\sigma)
\]
defines a simplex in \(\mathrm{VR}(X,r+4\varepsilon)\).
Indeed, this follows from the estimates
\begin{align*}
&\|x_i-g(f(x_i))\|\le\|x_i-f(x_i)\|+\|f(x_i)-g(f(x_i))\|\le2\varepsilon,\\
&\|x_i-g(f(x_j))\|\le\|x_i-x_j\|+\|x_j-f(x_j)\|+\|f(x_j)-g(f(x_j))\|\\
&\quad\le r+\varepsilon+\varepsilon=r+2\varepsilon,\\
&\|g(f(x_i))-g(f(x_j))\|\le\|g(f(x_i))-f(x_i)\|\\
&\quad+\|f(x_i)-f(x_j)\|+\|f(x_j)-g(f(x_j))\|\\
&\quad\le \varepsilon+(r+2\varepsilon)+\varepsilon=r+4\varepsilon.
\end{align*}
Therefore, the composite map
\[
G_{r+2\varepsilon}\circ F_r\colon\mathrm{VR}(X,r)\to\mathrm{VR}(X,r+4\varepsilon)
\]
is \emph{contiguous} to the inclusion map
\(\iota_{r+4\varepsilon,r}^{(X)}\) in \(\mathrm{VR}(X,r+4\varepsilon)\).
Similarly, the composite map
\[
F_{r+2\varepsilon}\circ G_r\colon\mathrm{VR}(Y,r)\to\mathrm{VR}(Y,r+4\varepsilon)
\]
is contiguous to the inclusion
\(\iota_{r+4\varepsilon,r}^{(Y)}\) in \(\mathrm{VR}(Y,r+4\varepsilon)\).

In general, two contiguous simplicial maps are homotopic on geometric
realizations and therefore induce identical maps on homology
\cite{edelsbrunner2010}.
Consequently, the homomorphism induced by
\(G_{r+2\varepsilon}\circ F_r\),
\[
(G_{r+2\varepsilon}\circ F_r)_*\colon
H_\ell(\mathrm{VR}(X,r))\to H_\ell(\mathrm{VR}(X,r+4\varepsilon)),
\]
coincides with
\[
(\iota_{r+4\varepsilon,r}^{(X)})_*\colon
H_\ell(\mathrm{VR}(X,r))\to H_\ell(\mathrm{VR}(X,r+4\varepsilon)).
\]
Similarly, the homomorphism induced by
\(F_{r+2\varepsilon}\circ G_r\),
\[
(F_{r+2\varepsilon}\circ G_r)_*\colon
H_\ell(\mathrm{VR}(Y,r))\to H_\ell(\mathrm{VR}(Y,r+4\varepsilon)),
\]
coincides with
\[
(\iota_{r+4\varepsilon,r}^{(Y)})_*\colon
H_\ell(\mathrm{VR}(Y,r))\to H_\ell(\mathrm{VR}(Y,r+4\varepsilon)).
\]

Hence, the persistence modules
\(\mathbb{X}:=\{H_\ell(\mathrm{VR}(X,r))\}_{r\ge0}\) and
\(\mathbb{Y}:=\{H_\ell(\mathrm{VR}(Y,r))\}_{r\ge0}\)
are \emph{\(2\varepsilon\)-interleaved}, and their interleaving distance
is at most \(2\varepsilon\).

Since Vietoris--Rips filtrations on finite point sets are finitely generated
in each degree, the persistence modules \(\mathbb{X}\) and \(\mathbb{Y}\)
are q-tame.
Therefore, by the stability theorem of persistence diagrams
\cite{cohen2007}, the bottleneck distance \(d_B\) and the interleaving
distance \(d_I\) satisfy
\[
d_B\bigl(D_\ell(\mathrm{VR}(X)),D_\ell(\mathrm{VR}(Y))\bigr)
\le d_I(\mathbb{X},\mathbb{Y}),
\]
and hence
\[
d_B\bigl(D_\ell(\mathrm{VR}(X)),D_\ell(\mathrm{VR}(Y))\bigr)
\le 2\varepsilon
=2\,d_H(X,Y).
\]
This completes the proof.
\end{proof}

\begin{remark}
The factor \(2\) arises from the diameter parametrization of the
Vietoris--Rips filtration.
For the \v{C}ech filtration or for a radius-normalized
Vietoris--Rips filtration, the corresponding stability bound holds with
constant \(1\).
\end{remark}

% ------------------------------------------
\begin{proof}[Proof of Theorem~\ref{thm:continuity_T}]
% For each simplex \(\sigma\in K^{(i)}:=2^{\{1,\ldots,N_i\}}\setminus\{\emptyset\}\) (\(i=1,2\) ) we have
% \[
% \Phi^{(i)}_\sigma(x^{(i)})=\max_{a,b\in\sigma}\|x_a^{(i)}-x_b^{(i)}\|.
% \]

% Using that the map \(x\mapsto\max_ix_i\) is 1-Lipschitz continuous with respect to \(\Vert\cdot\Vert_\infty\), together with the reverse triangle inequality \(|\|u\|-\|v\||\le \|u-v\|\), we have
% \begin{align*}
% \left|\Phi^{(i)}_\sigma(x^{(i)})-\Phi^{(i)}_\sigma(\tilde x^{(i)})\right|
% &\le
% \max_{a,b\in\sigma}\left|\|x_a^{(i)}-x_b^{(i)}\|
% -\|\tilde x_a^{(i)}-\tilde x_b^{(i)}\|\right| \\
% &\le
% \max_{a,b\in\sigma}\|(x_a^{(i)}-\tilde x_a^{(i)})
% -(x_b^{(i)}-\tilde x_b^{(i)})\| \\
% &\le
% 2d_{\mathcal{X}^{N_i}}(x^{(i)},\tilde x^{(i)}).
% \end{align*}
% Therefore
% \[
% \sup_{\sigma\in K^{(i)}}
% \left|\Phi^{(i)}_\sigma(x^{(i)})-\Phi^{(i)}_\sigma(\tilde x^{(i)})\right|
% \le
% 2d_{\mathcal{X}^{N_i}}(x^{(i)},\tilde x^{(i)}).
% \]

% By the stability theorem for persistence diagrams (see, e.g., \cite{cohen2007}), we have
% \[
% d_B\bigl(\mathcal{D}_\ell^{(N_i)}(x^{(i)}),\mathcal{D}_\ell^{(N_i)}(\tilde x^{(i)})\bigr)
% \le\sup_{\sigma\in K^{(i)}}\bigl|\Phi^{(i)}_\sigma(x^{(i)})-\Phi^{(i)}_\sigma(\tilde x^{(i)})\bigr|
% \le
% 2d_{\mathcal{X}^{N_i}}(x^{(i)},\tilde x^{(i)}).
% \]
Applying Proposition~\ref{prop:Continuity_Delta} and Proposition~\ref{prop:vr_stability} gives
\begin{align*}
&\left|T_\ell(x^{(1)},x^{(2)})-T_\ell(\tilde x^{(1)},\tilde x^{(2)})\right| \\
&=
\left|
\Delta_\ell\bigl(\mathcal{D}_\ell(x^{(1)}),\mathcal{D}_\ell(x^{(2)})\bigr)-
\Delta_\ell\bigl(\mathcal{D}_\ell(\tilde x^{(1)}),\mathcal{D}_\ell(\tilde x^{(2)})\bigr)
\right| \\
&\le
C_\Delta\sqrt{
d_B\bigl(\mathcal{D}_\ell(x^{(1)}),\mathcal{D}_\ell(\tilde x^{(1)})\bigr)
+d_B\bigl(\mathcal{D}_\ell(x^{(2)}),\mathcal{D}_\ell(\tilde x^{(2)})\bigr)}\\
&\le 
\sqrt{2}C_\Delta\sqrt{
d_H^{(N_1)}(x^{(1)},\tilde x^{(1)})
+d_H^{(N_2)}(x^{(2)},\tilde x^{(2)})
}.
\end{align*}
This proves the \(1/2\)-H\"{o}lder continuity of \(T_\ell\).
\end{proof}

% -----------------------------------------------
\begin{remark}[Practical choice of \(\gamma\) and \(\theta\)]
\label{rem:choice_gamma_theta}
The continuity constant in Theorem~\ref{thm:continuity_T} reads
\[
C_T=\sqrt{2S\,C_{\phi}\,C_{\mathrm{Density}}},
\]
\[
C_{\phi}=\frac{1+\theta}{(\gamma+\theta)\log 2}\!\left(1+\frac{2(1-\gamma)}{\theta}\right),
\quad
C_{\mathrm{Density}}=\frac{8(1-\gamma)}{S^2}\!\left(\frac{1}{\theta}+\frac{1}{\theta^2}\right).
\]
Moreover, Lemma~\ref{lem:continuity_of_density} yields uniform bounds
\(a_* \le p_\ell^{(\gamma,\theta)}(D)(t) \le a^*\) with
\(a_*=\tfrac{\gamma+\theta}{(1+\theta)S}\) and
\(a^*=\tfrac{1}{S}\bigl(1+\tfrac{2(1-\gamma)}{\theta}\bigr)\).

\medskip
\noindent\textit{Monotonicity and trade-offs.}
A direct calculus check shows that, on \((0,1)\), both \(C_{\phi}\) and \(C_{\mathrm{Density}}\) are
strictly decreasing in each of \(\gamma\) and \(\theta\).
Hence \(C_T\) decreases as either parameter increases (greater robustness),
while the lower bound \(a_*\) increases with \(\gamma\) or \(\theta\) (a stronger floor).
At the extremes:
\begin{itemize}
  \item \(\theta\to 0^+\): \(C_{\phi}\), \(C_{\mathrm{Density}}\) and \(C_T\) blow up; the statistic becomes numerically unstable and overly sensitive.
  \item \(\gamma\to 1^-\): \(p_\ell^{(\gamma,\theta)}(D)\to u_S\) and \(\Delta_\ell\equiv 0\); discrimination vanishes.
\end{itemize}

\medskip
\noindent\textit{Parameter management strategy.}
At the pre-processing stage, we set \(S=1\) by normalizing all attribute vectors with the maximum Euclidean distance among the remaining vectors (i.e., those after excluding outliers).
The parameter domains are set to \(\theta\in(0,1]\) and \(\gamma\in[0,1]\);
note that \(\gamma=1\) corresponds to the uniform density \(u_S\),
for which \(\Delta_\ell\equiv0\).
As an efficient rule for parameter management, we basically recommend fixing \(\gamma=0\) 
and controlling the regularization parameter \(\theta\) so that the density floor \(a_*\) 
remains above a specified threshold \(a_{\min}\).
Given the definition
\[
a_*=\frac{\theta}{1+\theta},
\]
the condition \(a_*\ge a_{\min}\) is equivalent to
\[
\theta\ge\frac{a_{\min}}{1-a_{\min}}=: \theta_{\min}.
\]
Therefore, by choosing \(\theta\ge\theta_{\min}\), one can maintain \(a_*\ge a_{\min}\) 
even with \(\gamma=0\).

In some situations, however, such as small-sample or highly concentrated data, 
the normalized mass \(\mu:=A_\ell(D)/A_{\max}\) may become very small.
To preserve sensitivity in such cases, it may be desirable to set \(\theta<\theta_{\min}\).
Then, to maintain \(a_*\ge a_{\min}\), a nonzero \(\gamma\) is required.
The condition
\[
a_*=\frac{\gamma+\theta}{1+\theta}\ge a_{\min}
\]
is equivalent to
\[
\gamma\ge a_{\min}-(1-a_{\min})\theta=:\gamma_{\min}.
\]
Hence, one can choose \(\gamma\ge\gamma_{\min}\) 
to ensure the desired floor level. 
To enhance discrimination, it is better to set \(\gamma\) to \(\gamma_{\min}\), but if you want to increase \(C_T\) and improve stability, you can achieve this by setting a larger \(\gamma\).

\medskip
\noindent
\textit{Numerical example.}
If \(a_{\min}=0.01\), then
\[
\theta_{\min}=\frac{0.01}{1-0.01}=0.\overline{01}.
\]
A simple practical recommendation is to use 
\((\gamma, \theta)=(0, 0.011)\) as a default configuration.
When a smaller \(\theta<0.\overline{01}\) is required,
one may estimate \(\mu\) (e.g., by the median of a held-out subset) 
and set \(\theta=\mu/10\) together with
\(\gamma\ge\gamma_{\min}=0.01-0.99\,\theta\)
so that \(a_*\ge a_{\min}\) is preserved.
This rule keeps the regularization minimal while ensuring numerical stability.
\end{remark}

% --------------------------------------------------
\section{Numerical Illustration: Dynamic Point Clouds}
\label{sec:numerical_illustration}
% -----------------------------------------------
Before proceeding to the numerical experiment, we briefly outline the motivation 
for analyzing the Gitcoin voting data from the perspective of topological data analysis (TDA).
The Gitcoin platform is part of the broader \emph{Web3} ecosystem, 
in which community governance is often implemented through 
\emph{decentralized autonomous organizations} (DAOs). 
Unlike traditional organizations with centralized administrators, 
a DAO operates on a blockchain and makes collective decisions 
through transparent, rule-based voting mechanisms encoded in smart contracts. 
Among these, \emph{quadratic voting} (QV) has become a key method: 
each participant is free to allocate votes across multiple candidates, 
but the cost of casting \(v\) votes increases quadratically (proportional to \(v^2\)). 
This design allows voters to express the \emph{intensity} of their preferences, 
while imposing a nonlinear budget constraint that encourages diversity of expression 
and prevents domination by large stakeholders.

The resulting voting data form high-dimensional vectors with rich nonlinear relationships 
between candidates, reflecting complex collective behaviors within the DAO community.
Detecting structural changes in such high-dimensional, nonlinear patterns 
is crucial for understanding shifts in community preferences or governance dynamics. 
Because these relationships are inherently geometric and topological in nature, 
TDA provides an appropriate framework for capturing their evolution. 
In particular, persistence-based statistics allow us to quantify 
changes in the topology—such as the emergence or disappearance of clusters or loops—
within the cloud of vote vectors over time.

In this section, we conduct empirical verification using the data from the
\emph{Gitcoin Steward Council Elections v3} (July 18--25, 2023). (See Appendix~\ref{sec:empirical_voting} for more information about this data.) 
The data come from a quadratic voting event conducted over eight days, during which multiple voters cast votes daily for a common set of ten candidates.
The ordered sequence of scores that each voter cast for each candidate is referred to as a "vote vector," and is treated as a 10-dimensional attribute vector.
Each day within the data period is treated as a separate time window,
resulting in eight windows in total.
For every pair of adjacent windows, we test whether the corresponding vote vector distributions are generated from the same underlying distribution,
following the Monte Carlo permutation-based procedure described in Section~\ref{sec:change_point_methods}.

\begin{description}
    \item[Null hypothesis]  
    The PL+JS statistics computed from the time windows \(\mathcal{W}_i\) and \(\mathcal{W}_{i+1}\) are generated from the same distribution.
    \item[Alternative hypothesis]  
    The PL+JS statistics computed from the time windows \(\mathcal{W}_{i}\) and \(\mathcal{W}_{i+1}\) are generated from different distributions.
\end{description}

The pre-processing and parameter settings follow the approach described in
Remark~\ref{rem:choice_gamma_theta}.
Specifically, at the pre-processing stage:
\begin{itemize}
  \item All pairwise Euclidean distances between vote vectors across all windows are collected.
  \item The 95th percentile of these distances is taken as \(q\). The remaining 5\% are discarded as outliers.
  \item Every vote vector is rescaled by dividing all its components by \(q\), thereby setting the global cutoff to \(S=1\).
\end{itemize}
The parameters are fixed to \((\gamma, \theta)=(0, 0.011)\).
The PL+JS statistic is computed for homological degrees \(\ell=0\) and \(\ell=1\),
each using \(M=3\) persistence landscape layers.

For each adjacent-window pair \((\mathcal{W}_i,\mathcal{W}_{i+1})\) we pool the two daily samples and
generate \(n=200\) random labelings that \emph{preserve the original group sizes}
\(\bigl(N_{i},N_{i+1}\bigr)\) (see Table~\ref{tab:voters_per_day}).
The Monte Carlo permutation \(p\)-value uses the standard “\(+1\) correction” \(\widehat p_n\) (\ref{eq:approximate_p_value}). By Proposition~\ref{prop:permutation_test_conservative},
these Monte Carlo permutation \(p\)-values are \emph{conservative}.
We test at a significance level of \(\alpha=0.05\).

\begin{table}[htbp]
  \centering
  \caption{Daily voter counts and total (Permutation test via Monte Carlo; 8 days)}
  \label{tab:voters_per_day}
  \begin{tabular}{l S}
    \toprule
    Date & {Number of voters} \\
    \midrule
    2023-07-18 & 350 \\
    2023-07-19 & 601 \\
    2023-07-20 & 197 \\
    2023-07-21 & 244 \\
    2023-07-22 & 268 \\
    2023-07-23 & 342 \\
    2023-07-24 & 443 \\
    2023-07-25 & 187 \\
    \midrule
    Total & 2632 \\
    \bottomrule
  \end{tabular}
\end{table}

\paragraph{Monte Carlo precision and empirical \(p\)-values.}
With \(n=200\) permutations, the worst–case standard error of \(\widehat p_n\) is
\(1/(2\sqrt{n+1})=1/(2\sqrt{201})\approx 0.035\) from Proposition~\ref{prop:mc-se}. Moreover, because
\(\widehat p_n=(1+K)/(n+1)\) with \(K\in\{0,\ldots,n\}\) the count of exceedances,
the attainable \(p\)-values lie on the grid
\(\{1/(n+1), 2/(n+1),\ldots,1\}\) with spacing (grid resolution)
\(\Delta=1/(n+1)=0.004975\). Thus decisions are typically stable unless
\(\widehat p_n\) lies very close to \(\alpha=0.05\).

In our adjacent-day tests (Table~\ref{tab:adjacent_day_shuffle}), the \(H_0\) (connected components)
\(p\)-values for \texttt{2023-07-20\(\to\)2023-07-21} and
\texttt{2023-07-21\(\to\)2023-07-22} equal the minimal nonzero value
\(\Delta=0.004975\)—i.e., none of the \(200\) shuffles produced a statistic
at least as large as the observed one—so these comparisons strongly reject at
\(\alpha=0.05\) and remain significant after Bonferroni adjustment for seven tests
(\(0.05/7\approx 0.007\)). The preceding pair \texttt{2023-07-19\(\to\)2023-07-20} in \(H_1\) has \(\widehat p_n=0.069652=14\Delta\), which is about \(4\Delta\) above \(\alpha=0.05\). Increasing the number of permutations would sharpen the estimate, but it would in any case remain non-significant under Bonferroni. As for other pairs, \(p\)-values are mostly well
above \(0.05\) (ranging from \(0.15\) to \(0.92\)), supporting the temporal stability.

\begin{table}[htbp]
  \centering
  \caption{Adjacent-day permutation tests (Monte Carlo approximation, \(n=200\)). 
  PL+JS statistics and permutation \(p\)-values for \(H_0\) and \(H_1\); 
  \(n_1,n_2\) are the daily sample sizes.}
  \label{tab:adjacent_day_shuffle}
  \begin{tabular}{
      l l
      S[table-format=1.6]
      S[table-format=1.6]
      S[table-format=1.6]
      S[table-format=1.6]
      S[table-format=3.0]
      S[table-format=3.0]}
    \toprule
    date\(_1\) & date\(_2\) & {\({PL+JS}_{H_0}\)} & {\(p_{H_0}\)} & {\({PL+JS}_{H_1}\)} & {\(p_{H_1}\)} & {\(n_1\)} & {\(n_2\)} \\
    \midrule
    2023-07-18 & 2023-07-19 & 0.103775 & 0.442786 & 0.264091 & 0.194030 & 350 & 601 \\
    2023-07-19 & 2023-07-20 & 0.219916 & 0.109453 & 0.275090 & 0.069652 & 601 & 197 \\
    2023-07-20 & 2023-07-21 & 0.328216 & 0.004975 & 0.089732 & 0.726368 & 197 & 244 \\
    2023-07-21 & 2023-07-22 & 0.436934 & 0.004975 & 0.090739 & 0.407960 & 244 & 268 \\
    2023-07-22 & 2023-07-23 & 0.108598 & 0.636816 & 0.030022 & 0.731343 & 268 & 342 \\
    2023-07-23 & 2023-07-24 & 0.363381 & 0.149254 & 0.035573 & 0.920398 & 342 & 443 \\
    2023-07-24 & 2023-07-25 & 0.179065 & 0.477612 & 0.027443 & 0.696517 & 443 & 187 \\
    \bottomrule
  \end{tabular}
\end{table}

\paragraph{Conclusion.}
Taken together, the day‐to‐day topology of vote vectors is largely stable. Only two
adjacent transitions—\texttt{2023-07-20\(\to\)2023-07-21} and \texttt{2023-07-21\(\to\)2023-07-22}—
show statistically significant differences in the \(H_0\) structure (after Bonferroni correction),
while no significant changes are detected for \(H_1\). Given the Monte Carlo grid resolution
\(\Delta=1/(n+1)=1/201\) and worst–case standard error \(\le 1/(2\sqrt{201})\approx0.035\),
these rejections reflect genuine effects rather than artifacts of simulation noise. We do not attempt to
assign causes here; explaining these shifts would require external covariates (e.g., turnout
composition or campaign events) and is left for future work.

As discussed in Remark~\ref{rem:multiple_testing}, 
multiple consecutive windows can also be tested jointly 
using the Holm step‐down procedure, 
which provides higher statistical power than the Bonferroni correction 
while still controlling the family-wise error rate.
In the present analysis, however, 
even the more conservative Bonferroni adjustment 
already yielded significant results for these two adjacent transitions.
This indicates that the detected changes are sufficiently strong 
to remain significant under stricter error control.

% -----------------------------------------------
\section{Discussion and Conclusion}
% -----------------------------------------------

\subsection{Summary of Contributions}
This work develops a probabilistic and computational framework for analyzing
persistence statistics in high-dimensional random point clouds.
We established moment bounds and tightness results for total and maximum
persistence under general distributions, and clarified their scaling behavior
within a Gaussian mixture setting.
Building on these theoretical foundations, we introduced a bounded,
distribution–free statistic---the PL+JS divergence---that combines
persistence landscapes with an information–theoretic distance.
We proved its \(1/2\)–H\"{o}lder continuity, providing a rigorous basis for persistence–based
comparisons of random point clouds.
Illustrative experiments on empirical data demonstrate that
the proposed framework can detect meaningful structural transitions in complex systems,
but the main contribution of this study lies in its theoretical underpinnings.

\subsection{Perspectives for Future Research}
Future work will extend the present analysis beyond Gaussian mixtures
to more general dependent or non–Euclidean settings,
and explore asymptotic regimes for random persistence under various sampling models.
Another direction concerns the integration of persistence–based divergences
into statistical inference pipelines, including hypothesis testing,
change detection, and confidence assessment.
From a computational standpoint, efficient approximation of persistence
landscapes and their gradients could enable extensions toward optimization
and machine–learning applications.
Finally, connecting the proposed statistic to interpretable geometric or
physical features remains an open challenge, bridging the gap between
probabilistic topology and data–driven analysis.

\backmatter
% \bmhead{Supplementary information}

\section*{Acknowledgements}
The author thanks Dr.~Yoshihiko Uchida, Professor at Shunan University, for valuable discussions.

\section*{Statements and Declarations}

\textbf{Competing Interests}\\
The author declares that there are no competing interests.

\medskip
\noindent
\textbf{Funding}\\
This research received no external funding.

\medskip
\noindent
\textbf{Data Availability}\\
The datasets used and/or analysed during the current study are available from the corresponding author on reasonable request.

% \bmhead{Ethics approval}
% Not applicable.

% \bmhead{Consent to participate}
% Not applicable.

% \bmhead{Consent for publication}
% Not applicable.

% \bmhead{Authors' contributions}

\begin{appendices}
\renewcommand{\thefigure}{\arabic{figure}}
\setcounter{figure}{0}  % the last number 

\renewcommand{\thetable}{\arabic{table}}
\setcounter{table}{3}  % the last number
% --------------------------------------------
\section{An Overview of TDA}
\label{sec:tda_basics}
% --------------------------------------------
This appendix section provides an overview of the basic concepts of TDA, 
a framework for extracting and quantifying the shape-related features of data, 
particularly geometric and topological properties such as cluster structures, loops, and voids as discussed in
\cite{carlsson2009, edelsbrunner2010, ghrist2008}.

% ----------------------------------------------
\subsection{Simplicial Complexes}
To endow a discrete set of points with a topological structure, a \emph{simplicial complex} is used.  
A simplicial complex is a finite set of simplices—points (0-simplices), edges (1-simplices), triangles (2-simplices), tetrahedra (3-simplices), etc.—combined in a way that respects inclusion relations.

Given a discrete point cloud \(x=(x_1,\ldots,x_N)\in(\mathbb{R}^d)^N\),
we can construct a family of simplicial complexes parameterized by a scale parameter \(r>0\).
For each fixed \(r\), two representative constructions are the \v{C}ech and Vietoris–Rips complexes, which define the connectivity structure of \(X\) as scale \(r\).
\begin{description}
    \item[\emph{\v{C}ech complex}]  
    For each data point, consider an open ball of radius \(r/2\) centered at that point.  
    A simplex is formed for any set of points whose corresponding balls have a nonempty common intersection.  
    Here, a \(k\)-simplex corresponds to a set of \(k+1\) points.  
    The choice of \(r/2\) ensures that the \v{C}ech complex and the Vietoris–Rips complex (defined below) are built at comparable geometric scales.  
    The \v{C}ech complex is known to faithfully capture the topological features of the underlying space (nerve theorem).
    \item[\emph{Vietoris–Rips complex}]  
    A \(k\)-simplex is formed if all pairwise distances among a set of \(k+1\) points are at most \(r\).  
    Due to its computational efficiency, the Vietoris–Rips complex is often preferred in practice.
\end{description}

% --------------------------------------------
\subsection{Homology}
To extract topological features from a simplicial complex \(K\), \emph{homology} is used.  

The space of \(\ell\)-\emph{chains} over the field \(\mathbb{F} := \mathbb{Z}/2\mathbb{Z}\) is defined as:
\[
C_\ell(K) := \bigoplus_{\sigma \in K,\, \operatorname{dim} \sigma = \ell} \mathbb{F}\sigma,
\]
where the direct sum is taken over all \(\ell\)-simplices of \(K\).

The \emph{boundary} homomorphism \(\partial_\ell \colon C_\ell(K) \to C_{\ell-1}(K)\) is defined for  
\(\sigma = \{v_0, v_1, \ldots, v_\ell\} \in K\) by:
\[
\partial_\ell \sigma = \sum_{i=0}^\ell \{v_0, \ldots, \hat{v_i}, \ldots, v_\ell\} \in C_{\ell-1}(K),
\]
where \(\hat{\cdot}\) denotes omission.

The \emph{cycle group}, \emph{boundary group}, and \emph{homology group} are given by:
\begin{align*}
Z_\ell(K) &:= \operatorname{Ker} \partial_\ell \subset C_\ell(K), \\
B_\ell(K) &:= \operatorname{Im} \partial_{\ell+1} \subset Z_\ell(K), \\
H_\ell(K) &:= Z_\ell(K) / B_\ell(K).
\end{align*}

The \(\ell\)-th homology group \(H_\ell(K)\) characterizes \(\ell\)-dimensional “holes” (connected components, loops, voids, etc.) in the simplicial complex \(K\).
These homology groups form the basis of our topological analysis, as their changing structures directly reflect changes in the underlying geometry of the data.
For clarity regarding the geometric and statistical meaning of these features, we adopt the following standard interpretations throughout this paper.
\begin{remark}[Interpretation of Homology Elements]
For interpretation in the context of statistical data analysis and application to high-dimensional attribute vectors, we use the following standard geometric interpretations:
\begin{enumerate}
\item \(H_0(K)\): Elements of the \(0\)-th homology group are interpreted as \textbf{connected components}. Statistically, these are viewed as \textbf{clusters} or distinct groupings within the data.
\item \(H_1(K)\): Elements of the \(1\)-st homology group are interpreted as \textbf{loops} or circular structures in the data.
\item \(H_2(K)\): Elements of the \(2\)-nd homology group are interpreted as \textbf{voids} (cavities) or \(3\)-dimensional holes.
\end{enumerate}
\end{remark}

% ---------------------------------------
\subsection{Persistence Diagram}
\label{sec:persistence_diagram}
Let \(d, N \in \mathbb{N}\) and let \(x = (x_1, \ldots, x_N) \in (\mathbb{R}^d)^N\) 
be a finite point cloud in \(\mathbb{R}^d\), where each element of \(\mathbb{R}^d\) represents an attribute vector. 

The set \(K = 2^{\{1,\ldots,N\}} \setminus \{\emptyset\}\) corresponds to the simplicial complex consisting of all faces of the \((N-1)\)-simplex.

A \emph{filtration} is a map
\[
\Phi = (\Phi_\sigma)_{\sigma \in K} \in \mathbb{R}^{\lvert K \rvert}
\]
such that \(\Phi_{\sigma_1} \le \Phi_{\sigma_2}\) whenever \(\sigma_1 \subset \sigma_2\) (where \(\le\) denotes the standard order on \(\mathbb{R}\)). 
Here, \(\lvert K \rvert\) denotes the cardinality of the set \(K\).

For a given real parameter \(r\), we denote by 
\[
K_r = \{\sigma \in K :\,  \Phi_\sigma \le r\}
\]
the subcomplex consisting of all simplices whose filtration value is at most \(r\).

For each homological degree \(\ell \in \mathbb{N}_{0}\), there exist an integer \(m \in \mathbb{N}_{0}\),
a set of \(\ell\)-chains \(\{ z_i \in C_\ell(K) :\,  i = 1, \ldots, m \}\),
and a sequence \(((b_i, d_i))_{i=1,\ldots,m}\) (\(0 \le b_i \le d_i \le \infty\))
such that:
\begin{enumerate}
    \item If \(r < b_i\), then \(z_i \notin Z_\ell(K_r)\).
    \item If \(r \ge b_i\), then \(z_i \in Z_\ell(K_r)\).
    \item If \(b_i \le r < d_i\), then \(z_i \notin B_\ell(K_r)\).
    \item If \(r \ge d_i\), then \(z_i \in B_\ell(K_r)\).
    \item For each \(i\) with \(b_i \le r < d_i\), the homology classes of \(z_i\) in \(H_\ell(K_r)\) form a basis of \(H_\ell(K_r)\).
\end{enumerate}

\begin{remark}
This is classical; see \cite{edelsbrunner2002,edelsbrunner2010,ghrist2008}.
For stability of barcodes under perturbations of the filtration, see \cite{cohen2007}.
\end{remark}

Each pair \((b_i, d_i)\) may have multiplicity and is referred to as a \emph{birth–death pair} (or bar). 
The birth–death pairs \(((b_i, d_i))_{i=1,\ldots,m}\) are regarded as a locally finite multisets 
on \(\{(b,d)\in\overline{\mathbb{R}}^2: b<d\}\) with the convention that the diagonal 
\(\{(t,t):t\in\overline{\mathbb{R}}\}\) has infinite multiplicity. 
The finite multiset \(D\) is called the \emph{persistence diagram} (or barcode) in degree \(\ell\) 
(see, e.g., \cite{edelsbrunner2002, cohen2007}).

% ----------------------------------------------------
\subsection{Bottleneck Distance}\label{sec:bottleneck_distance}
% ----------------------------------------------------

The bottleneck distance was introduced by \cite{cohen2007}, a standard metric on persistence diagrams that quantifies the largest deviation between features of two diagrams under optimal matching.
This distance plays a central role in stability results for persistent homology.

In what follows, we reuse the notation introduced in Section~\ref{sec:general_preliminaries}.
\begin{definition}[Bottleneck Distance]
Let \(D, D' \in \Dgm\) be persistence diagrams in degree \(\ell\). 

A \emph{matching} \(\Gamma\) between \(D\) and \(D'\) is a subset
\[
\Gamma \subset D \times D'
\]
such that the coordinate projections
\[
\pi_1 : \Gamma \to D, \quad \pi_2 : \Gamma \to D'
\]
are bijections onto their respective diagrams after adding diagonal points. 
Unmatched points are allowed to be paired with point in the diagonal (\(\Delta:=\{(b,d)\in\bar{\mathbb{R}}^2:\,  b=d\}\).

The \emph{cost} of \(\Gamma\) is defined by
\[
\operatorname{cost}(\Gamma) :=
\sup_{(p,q)\in\Gamma} \|p - q\|_\infty,
\quad
\|(b,d)\|_\infty := \max\{|b|,\,|d|\}.
\]

The \emph{bottleneck distance} in degree \(\ell\) is
\[
d_{\ell, B}\bigl(D, D'\bigr)
:=
\inf_{\Gamma} \operatorname{cost}(\Gamma),
\]
where the infimum runs over all matchings~\(\Gamma\).
\end{definition}

\begin{remark}
Infinite multiplicity on \(\Delta\) ensures that a matching always exists by allowing any leftover point
to be paired with a diagonal point (interpreting zero-length features as noise).
\end{remark}

% ------------------------------------------------------------------
\subsection{Persistence Landscape}\label{sec:PL}
% ------------------------------------------------------------------
Persistence diagrams provide a finite multiset representation of topological features, but for certain analytical and statistical purposes, it is often more convenient to work with a functional summary.
The \emph{persistence landscape}, introduced by \cite{bubenik2015}, transforms a persistence diagram into a sequence of piecewise-linear functions, enabling the use of functional analysis and statistical tools while retaining key topological information.
We recall its definition below, following the notation in Section~\ref{sec:general_preliminaries}.

For each point \((b,d)\) (\(d<\infty\)) in a persistence diagram \(D\) in degree \(\ell\in\mathbb{N}_0\), we define a triangular-shaped function supported on the interval \([b,d]\), with its maximum value attained at the midpoint \(\frac{b+d}{2}\), as follows:
\[
f_{(b,d)}(t) = \begin{cases}
t - b & \text{if } b \le t \le \frac{b+d}{2}, \\
d - t & \text{if } \frac{b+d}{2} < t \le d, \\
0     & \text{otherwise}
\end{cases}
\]
This function takes positive values only on the interval \([b,d]\) and forms a symmetric triangular function whose peak height is \(\frac{d - b}{2}\).

For any \(t\in \mathbb{R}\), let the \(m\)-th largest value be 
\[
\lambda_{\ell,m}(D)(t)=\operatorname{\text{\(m\)-}\max}
\{f_{(b,d)}(t):\,  (b,d) \in D\}.
\]
Here note that \(D\) has finite elements. 
The function \(\lambda_{\ell, m}(D)\) is called the \(m\)-th persistence landscape layer, and the sequence \((\lambda_{\ell, m}(D))_{m\in\mathbb{N}_0}\) is referred to as the \emph{persistence landscape}.

% -------------------------------------------------------
\section{JS Divergence}\label{sec:js_divergence}
% -------------------------------------------------------

For completeness, we briefly recall the definition of the Jensen–Shannon (JS) divergence,
which is used in the construction of the PL+JS statistic in the main text. 
The JS divergence is a symmetrized and smoothed variant of the Kullback–Leibler (KL) divergence, and is widely used as a finite and symmetric 
measure of dissimilarity between probability distributions. 
Throughout this paper, we compute divergences using logarithm base~\(2\). 
This choice is twofold: first, it expresses the divergence in \emph{bits}, 
which is standard in information theory; second, it normalizes the 
maximum possible JS divergence to~\(1\), which makes the JS distance 
lie in the range \([0,1]\) and facilitates interpretation. 
We begin by recalling the definition of the KL divergence, 
which quantifies the discrepancy between two probability density functions 
in an asymmetric manner.

\begin{definition}
Let \(p\)  and \(q\)  be probability density functions with respect to a common reference measure~\(\mu\).
The \emph{Kullback-Leibler (KL) divergence} of \(p\) from \(q\), computed with logarithm base~\(2\), is defined as
\[
  \mathrm{KL}(p \Vert q)
  := \int_{\{\,q(t)>0\,\}} p(t)\,\log_2 \frac{p(t)}{q(t)}\,d\mu(t).
\]
We adopt the following conventions for the integrand:
\begin{itemize}
  \item If \(p(t)=0\), then \(p(t)\log_2 \frac{p(t)}{q(t)}\) is taken to be \(0\).
  \item If \(q(t)=0\) and \(p(t)>0\) on a set of positive \(\mu\) -measure, then \(\mathrm{KL}(p\Vert q)=+\infty\).
\end{itemize}
\end{definition}

\begin{definition}
Let \(p\)  and \(q\)  be probability density functions.  
The \emph{Jensen-Shannon (JS) divergence} is defined as
\[
  \mathrm{JS}(p \Vert q)
  := \tfrac{1}{2}\,\mathrm{KL}(p \Vert m) 
   + \tfrac{1}{2}\,\mathrm{KL}(q \Vert m),
\]
where \(m\)  is the density of the mixture distribution
\[
  m(t) = \tfrac{1}{2}\bigl(p(t)+q(t)\bigr).
\]
Since \(m(t)>0\)  whenever \(p(t)>0\)  or \(q(t)>0\) , both terms in the definition are finite, and hence \(\mathrm{JS}(p\Vert q)<\infty\) .
\end{definition}

\begin{definition}
For two probability density functions \(p\) and \(q\),  
the Jensen–Shannon distance, computed with logarithm base \(2\), is defined as
\[
d_{\text{JS}}(p,q)=\sqrt{D_{\text{JS}}(p\Vert q)}.
\]
\end{definition}

% ----------------------------------------
\section{Empirical Analysis: A Case Study on Web3 Voting Behavior}
\label{sec:empirical_voting}
% ----------------------------------------
In appendix section, we describe the preprocessing procedure applied to the 
\emph{Gitcoin Steward Council Elections v3} dataset used in Section~\ref{sec:numerical_illustration}. 
We then clarify the characteristics of the data through principal component analysis (PCA) 
and Gaussian mixture model (GMM) fitting, and finally conduct simulations 
to empirically verify the validity of Corollary~\ref{cor:uniform_poly_tail}.

\subsection{Dataset and Methods}
This study analyzes the voting data from the Gitcoin Steward Council Elections v3, held from July 18 to July 25, 2023.
The data were collected via the “GraphQL API” provided by Snapshot, a decentralized voting platform, by accessing \url{https://hub.snapshot.org/graphql}.  
We searched for the proposal titled ``Gitcoin Steward Council Elections v3" submitted by the community \texttt{gitcoindao.eth}, and developed a Python script to automatically retrieve all voting records associated with this proposal.

Each voting record contains the voter’s address, timestamp, and the score allocation for each candidate.  
The dataset, originally in nested JSON format, was transformed into a flat tabular structure using \texttt{pandas.json\_normalize()} and subsequently stored in CSV format for further analysis.  
It covers an eight-day period, from July 18 to July 25, 2023, and consists of 2,632 records. 
The columns \texttt{choice.1} through \texttt{choice.10} represent the scores assigned by each voter to each candidate, and the sequence of these ten scores is referred to as the vote vector, indicating how each voter distributed their scores among multiple candidates.

Our analytical procedure proceeds in two steps:  
(1) As a preliminary step before conducting TDA, we initially apply principal component analysis (PCA) to examine the contribution of the linear components of the data, evaluating the cumulative explained variance and principal components.
(2) Fit a GMM to capture latent clusters of voting behavior, and conduct simulations by controlling the variance scaling parameter \(\eta\).  
This enables us to examine how persistence statistics (total and maximum persistence) evolve as dispersion grows, linking directly to the theoretical results on variance scaling.

\subsection{PCA}
We first center each voter’s 10-dimensional vector of scores for the candidates  
(by subtracting the mean vector shown in Table~\ref{tab:center-values})  
and apply principal component analysis (PCA).  
The results are presented in Tables~\ref{tab:eigenvalue-choices} and \ref{tab:eigenvectors}.

According to the mean vector in Table~\ref{tab:center-values},  
Candidate~1 has the highest mean score (approximately 4.29),  
significantly exceeding the others,  
followed by Candidate~2 (1.34) and Candidate~4 (1.17).  
This indicates that, overall, there is strong support for Candidate~1,  
while Candidates~2 and~4 also receive a non-negligible share of votes.

From the eigenvalues and variance ratios in Table~\ref{tab:eigenvalue-choices},  
the proportion of variance explained by the first principal component is about 58\%,  
followed by 12\% for the second, 7.7\% for the third, and 6.2\% for the fourth.  
The cumulative contribution rate reaches 78\% with three components and 91\% with six components.  

The first eigenvector (PC1) has a very large positive loading for Candidate~1 (about 0.99),  
while the loadings for other candidates are all 0.13 or less (Table~\ref{tab:eigenvectors}).  
Thus, the first principal component strongly reflects  
the magnitude of the score for Candidate~1,  
with higher PC1 scores corresponding to voters who concentrate their scores on Candidate~1.

The second eigenvector (PC2) has a large positive loading for Candidate~4 (about 0.77),  
and relatively large values for Candidates~2,~5,~3 and~6,  
while Candidate~1 has a negative loading (-0.11).  
Hence, the second principal component represents an axis contrasting  
“emphasis on Candidates~4, 2, and 5” versus “emphasis on Candidate~1.”

The third eigenvector (PC3) has the largest positive loading for Candidate~2 (about 0.80)  
and a large negative loading for Candidate~4 (-0.52).  
This axis represents another pattern contrasting “emphasis on Candidate~2” versus “emphasis on Candidate~4.”

In summary, the voting structure of this dataset can be characterized by three main axes:
\begin{itemize}
  \item PC1: Concentration on Candidate~1
  \item PC2: Candidates~4 (and 2, 5) vs Candidate~1
  \item PC3: Candidate~2 vs Candidate~4
\end{itemize}

The patterns suggested by PC1–PC3 indicate distinct archetypes of scoring behavior.
To operationalize this structure, we model the vote vectors as a GMM, which may provide  probabilistic cluster memberships aligned with these axes.

\begin{table}[htbp]
\centering
\caption{Mean Vector of Voting Scores}
\label{tab:center-values}
\begin{tabular}{lr}
\hline
Candidate & Mean Score \\
\hline
choice.1 & 4.29 \\
choice.2 & 1.34 \\
choice.3 & 0.88 \\
choice.4 & 1.17 \\
choice.5 & 0.73 \\
choice.6 & 0.66 \\
choice.7 & 0.51 \\
choice.8 & 0.41 \\
choice.9 & 0.46 \\
choice.10 & 0.50 \\
\hline
\end{tabular}
\end{table}

\begin{table}[htbp]
\centering
\caption{Eigenvalues and Variance Ratios (PCA Results)}
\label{tab:eigenvalue-choices}
\begin{tabular}{lrrr}
\hline
Principal Component & Eigenvalue & Variance Ratio & Cumulative \\
\hline
PC1 & 222.38 & 58\% & 58\% \\
PC2 & 45.93 & 12\% & 70\% \\
PC3 & 29.47 & 7.7\% & 78\% \\
PC4 & 23.68 & 6.2\% & 84\% \\
PC5 & 15.64 & 4.1\% & 88\% \\
PC6 & 12.10 & 3.2\% & 91\% \\
PC7 & 10.39 & 2.7\% & 94\% \\
PC8 & 9.18 & 2.4\% & 96\% \\
PC9 & 8.71 & 2.3\% & 99\% \\
PC10 & 4.99 & 1.4\% & 100\% \\
\hline
\end{tabular}
\end{table}

\begin{table}[htbp]
\centering
\caption{Principal Component Loadings}
\label{tab:eigenvectors}
\setlength{\tabcolsep}{2pt}  
\begin{tabular}{lrrrrrrrrrr}
\hline
 & choice.1 & choice.2 & choice.3 & choice.4 & choice.5 & choice.6 & choice.7 & choice.8 & choice.9 & choice.10 \\
\hline
PC1  & 0.99 & 0.13 & 0.05 & 0.05 & 0.01 & 0.00 & 0.01 & 0.01 & 0.02 & 0.02 \\
PC2  & -0.11 & 0.38 & 0.24 & 0.77 & 0.30 & 0.24 & 0.15 & 0.08 & 0.08 & 0.10 \\
PC3  & -0.10 & 0.80 & 0.27 & -0.52 & 0.06 & -0.02 & -0.01 & 0.03 & 0.02 & 0.04 \\
PC4  & 0.03 & -0.29 & 0.08 & -0.35 & 0.47 & 0.52 & 0.32 & 0.23 & 0.34 & 0.08 \\
PC5  & 0.00 & -0.25 & 0.81 & -0.02 & -0.39 & 0.28 & -0.20 & -0.10 & -0.07 & 0.03 \\
PC6  & -0.02 & 0.05 & -0.08 & 0.05 & -0.54 & -0.13 & 0.32 & 0.05 & 0.60 & 0.46 \\
PC7  & 0.00 & -0.21 & 0.34 & -0.03 & 0.41 & -0.64 & -0.02 & 0.14 & -0.08 & 0.48 \\
PC8  & 0.01 & 0.05 & -0.27 & -0.05 & -0.06 & 0.39 & -0.14 & -0.06 & -0.48 & 0.72 \\
PC9  & -0.01 & 0.03 & -0.12 & 0.04 & 0.08 & 0.08 & -0.83 & 0.30 & 0.43 & 0.08 \\
PC10 & 0.00 & 0.01 & 0.01 & 0.02 & -0.26 & -0.02 & 0.14 & 0.90 & -0.30 & -0.11 \\
\hline
\end{tabular}
\end{table}

% -------------------------------------------
\subsection{GMM}
To investigate how the characteristics of voting behavior affect the topology,  
we construct a simulation model of the voting behavior.

The Gitcoin Steward Council Elections v3 voting data consist of  
Quadratic Voting for \(d=10\) candidates by \(N=2632\) voters.  
The vote vector of voter \(i\), representing the number of votes for each candidate, is given by
\[
x_i = (x_{i1}, x_{i2}, \dots, x_{i,10})^\top 
\in \mathbb{R}^{10}, \quad i = 1,\dots,N.
\]

We assume that these vote vectors are independent samples drawn from  
a \(K\)-component GMM in 10-dimensional space \(\mathbb{R}^{10}\),  
where \(K\) is the number of clusters.  
As we have seen in Section~\ref{sec:gmm}, which defines the GMM,  
the parameters are \(\{\pi_k,\mu_k,\Sigma_k\}_{k=1}^K\).  
We estimate these parameters by maximum likelihood (EM algorithm) using scikit-learn’s \texttt{GaussianMixture} with \emph{full covariance matrices}  
(i.e., each component has its own unrestricted positive-definite covariance matrix, allowing arbitrary cluster shapes and orientations)  
and \texttt{n\_init = 10} random initializations to mitigate the risk of convergence to poor local optima.

To determine the number of components \(K\), we evaluate candidate values \(K = 1, \dots, 10\),  
fitting a GMM for each and computing the Bayesian and Akaike information criteria (BIC/AIC).  
Although both BIC and AIC continue to decrease slightly beyond \(K=6\),  
we adopt \(K=6\) as a parsimonious choice that balances model fit and simplicity.  
As a sanity check, we generate synthetic samples from the fitted \(K=6\) model using its estimated 
\(\{\pi_k, \mu_k, \Sigma_k\}\) and repeat the same BIC/AIC evaluation; 
the resulting profiles closely align with those from the real data (see Figure~\ref{fig:gmm_model_selection}).
Here, ``real'' refers to the observed vote vector data.  
The number of synthetic samples in each cluster is determined according to \(\pi_k\),  
and the samples are drawn from the multivariate normal distribution \(\mathcal{N}(\mu_k, \Sigma_k)\).  
Comparing the BIC/AIC curves for real and synthetic data allows us to verify whether the fitted GMM reproduces the cluster structure of the real data; in our case, both exhibit similar decreasing trends with a clear elbow at \(K=6\).

\begin{figure}[htbp]
    \centering
    \includegraphics[width=0.85\linewidth]{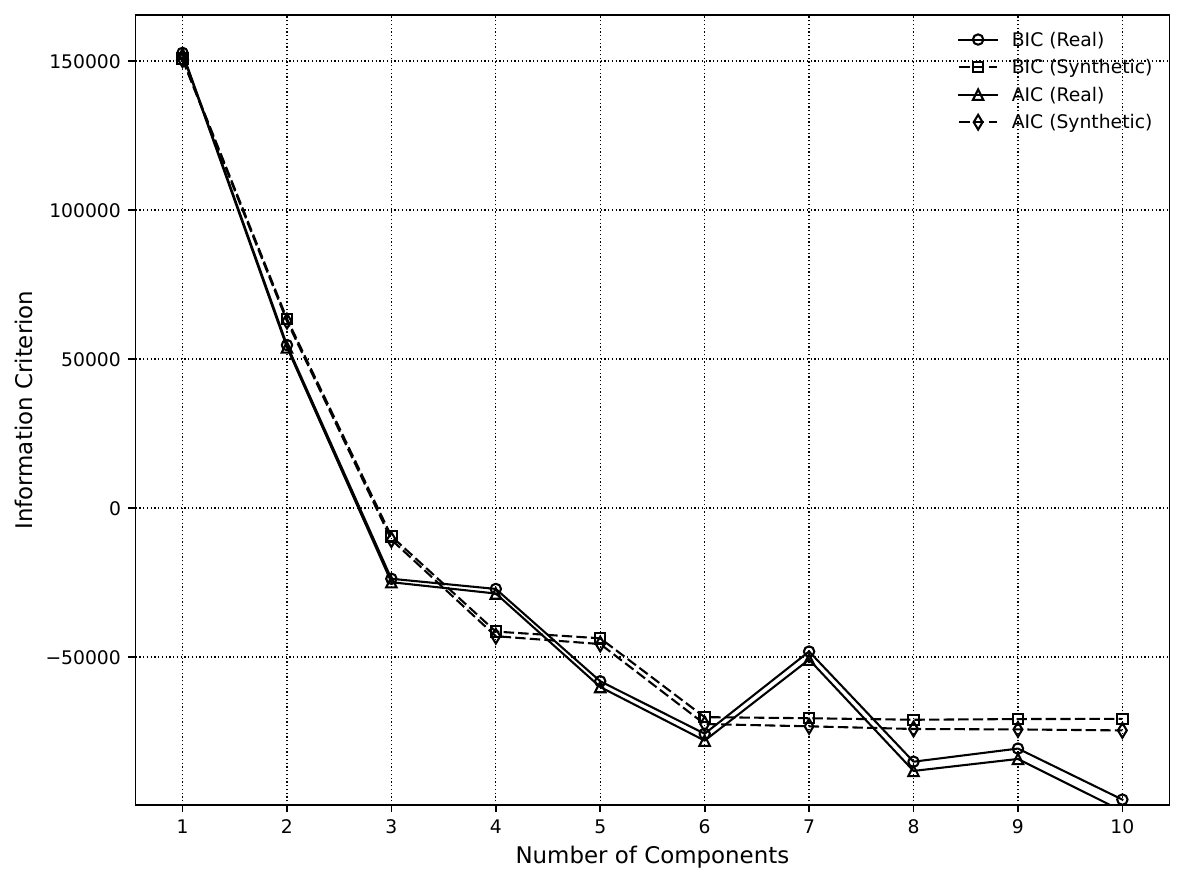}
    \caption{Comparison of cluster number selection by BIC and AIC (Real vs Synthetic)}
    \label{fig:gmm_model_selection}
\end{figure}

The final model selected is a GMM with \(K=6\) components.
Tables~\ref{tab:gmm_components_1} and \ref{tab:gmm_components_2} 
summarize the estimated parameters of each component. 
For readability, the results are split into two tables: 
Table~\ref{tab:gmm_components_1} reports the square roots of the 
minimum
\footnote{Note that the entries shown as 0.0 are not exactly zero but small positive values.
The GaussianMixture implementation in scikit-learn applies internal regularization and numerical stabilization, ensuring that the covariance matrices remain positive definite and effectively preventing true zero (singular) eigenvalues. See Remark~\ref{rem:key_features_of_gmm} (i).}
and maximum eigenvalues of the covariance matrices, 
the mixture weights, and the first four entries of the mean vector, 
while Table~\ref{tab:gmm_components_2} shows the remaining mean vector entries.
Components 2, 5, and 6 are nearly degenerate, with very small minimum eigenvalues, indicating that the distribution is effectively confined to a lower-dimensional subspace.
The coexistence of nearly degenerate components (2, 5, and 6) and more diffuse components (1, 3, and 4) suggests heterogeneity in decision-making styles: 

\begin{table}[htbp]
\centering
\caption{Summary Statistics of GMM Components -- 1}
\label{tab:gmm_components_1}
\begin{tabular}{cccccccc}
\toprule
 Component & sqrt\_min & sqrt\_max & weight & choice.1 & choice.2 & choice.3 & choice.4 \\
\midrule
 1 & 0.4 & 1.0 & 0.28 & 0.6 & 0.5 & 0.4 & 0.5 \\
 2 & 0.0 & 20.5 & 0.34 & 6.3 & 0.0 & 0.0 & 0.0 \\
 3 & 0.7 & 11.0 & 0.05 & 7.1 & 3.5 & 4.2 & 2.3 \\
 4 & 5.6 & 25.0 & 0.02 & 19.1 & 6.8 & 3.1 & 6.5 \\
 5 & 0.0 & 34.0 & 0.03 & 25.6 & 16.0 & 12.1 & 21.8 \\
 6 & 0.0 & 8.1 & 0.27 & 1.6 & 1.5 & 0.5 & 0.5 \\
\bottomrule
\end{tabular}
\end{table}

\begin{table}[htbp]
\centering
\caption{Summary Statistics of GMM Components -- 2}
\label{tab:gmm_components_2}
\begin{tabular}{ccccccc}
\toprule
 Cluster & choice.5 & choice.6 & choice.7 & choice.8 & choice.9 & choice.10 \\
\midrule
 1 & 0.4 & 0.5 & 0.4 & 0.4 & 0.4 & 0.5 \\
 2 & 0.0 & 0.0 & 0.0 & 0.0 & 0.0 & 0.0 \\
 3 & 2.2 & 2.0 & 1.7 & 1.9 & 1.8 & 2.5 \\
 4 & 7.0 & 6.9 & 5.9 & 7.8 & 10.4 & 9.6 \\
 5 & 10.6 & 8.6 & 5.4 & 0.1 & 0.0 & 0.1 \\
 6 & 0.0 & 0.0 & 0.0 & 0.0 & 0.0 & 0.0 \\
\bottomrule
\end{tabular}
\end{table}

% ---------------------------------------------------------
\subsection{Effect of Variance Scaling on Model Fit and persistence statistics}
\label{sec:effect_of_variance_scaling}

We now investigate how scaling the covariance matrices in the fitted \(K=6\) GMM affects both model fit and the topology of the generated point clouds.  
For a given scale factor \(\eta>0\), the covariance matrices \(\Sigma_k\) are replaced with \(\eta \Sigma_k\), and synthetic samples are drawn from the resulting model.  

From the topological perspective, we focus on four quantities derived from the definitions in Section~\ref{sec:general_preliminaries} and the mathematical results in Theorem~\ref{thm:lp_estimation_gmm} or Corollary~\ref{cor:uniform_poly_tail}:
\begin{itemize}
    \item \emph{Sum H0 Bar}: the total persistence of \(H_0\) bars (\(\mathcal{L}_{0, \mathrm{total}}\)).
    \item \emph{Sum H1 Bar}: the total persistence of \(H_1\) bars (\(\mathcal{L}_{1, \mathrm{total}}\)).
    \item \emph{Max H0 Bar}: the maximum persistence of \(H_0\) bars (\(\mathcal{L}_{0, \mathrm{max}}\)).
    \item \emph{Max H1 Bar}: the maximum persistence of \(H_1\) bars (\(\mathcal{L}_{1, \mathrm{max}}\)).
\end{itemize}

Figures~\ref{fig:scale_factor_vs_metrics_h0} and \ref{fig:scale_factor_vs_metrics_h1} plot these persistence statistics together with the mean log-likelihood as functions of the scale factor~\(\eta\).
Here, by “log-likelihood’’ we mean the pointwise log-density of the fitted GMM
\(\theta=\{\pi_k,\mu_k,\Sigma_k\}_{k=1}^K\) evaluated at a data point \(x\in\mathbb{R}^{10}\):
\[
  f(\theta; x)
  := \log\left(\sum_{k=1}^K \pi_k \,\varphi(x\mid \mu_k,\Sigma_k)\right),
\]
here \(\varphi(\cdot\mid\mu_k,\Sigma_k)\) is the Gaussian density with mean \(\mu_k\) and covariance \(\Sigma_k\).

For each \(\eta\), we generate synthetic samples \(x'_1(\eta),\ldots,x'_{N}(\eta)\) 
from the fitted GMM with covariances scaled by a factor \(\eta\)
(i.e., \(\Sigma_k \mapsto \eta \Sigma_k\) while keeping \((\pi_k,\mu_k)\) fixed), using the same sample size \(N\) as the original dataset, and we report the mean log-likelihood
\[
  \bar f(\eta) \;=\; \frac1N \sum_{j=1}^N f(\theta; x'_j(\eta)).
\]
In Figures~\ref{fig:scale_factor_vs_metrics_h0} and \ref{fig:scale_factor_vs_metrics_h1}, this quantity \(\bar f(\eta)\) is plotted.

The mean log-likelihood decreases almost monotonically as \(\eta\) increases. This is expected: enlarging each component’s covariance draws more dispersed samples, which lowers the predictive density under the fixed fitted model. 

Regarding the persistence statistics, the results are consistent with the bounds in Theorem~\ref{thm:lp_estimation_gmm} or Corollary~\ref{cor:uniform_poly_tail}.  
In particular:
\begin{itemize}
    \item \emph{Sum H0 Bar and Max H0 Bar} increases rapidly for small \(\eta\) and then plateaus, reflecting the \(\mathcal{O}(\eta^{1/2})\) growth predicted by Corollary~\ref{cor:uniform_poly_tail}.
    \item \emph{Sum H1 Bar and Max H1 Bar} exhibits larger fluctuations, especially for small \(\eta\), consistent with the sensitivity of \(H_1\) features to noise and the sparser appearance of loops in the data.
\end{itemize}

Overall, the empirical trends of both \(H_0\) and \(H_1\) metrics align qualitatively with the \(\eta^{1/2}\)-order bounds in Corollary~\ref{cor:uniform_poly_tail}.

\begin{figure}[htbp]
    \centering
    \includegraphics[width=0.85\linewidth]{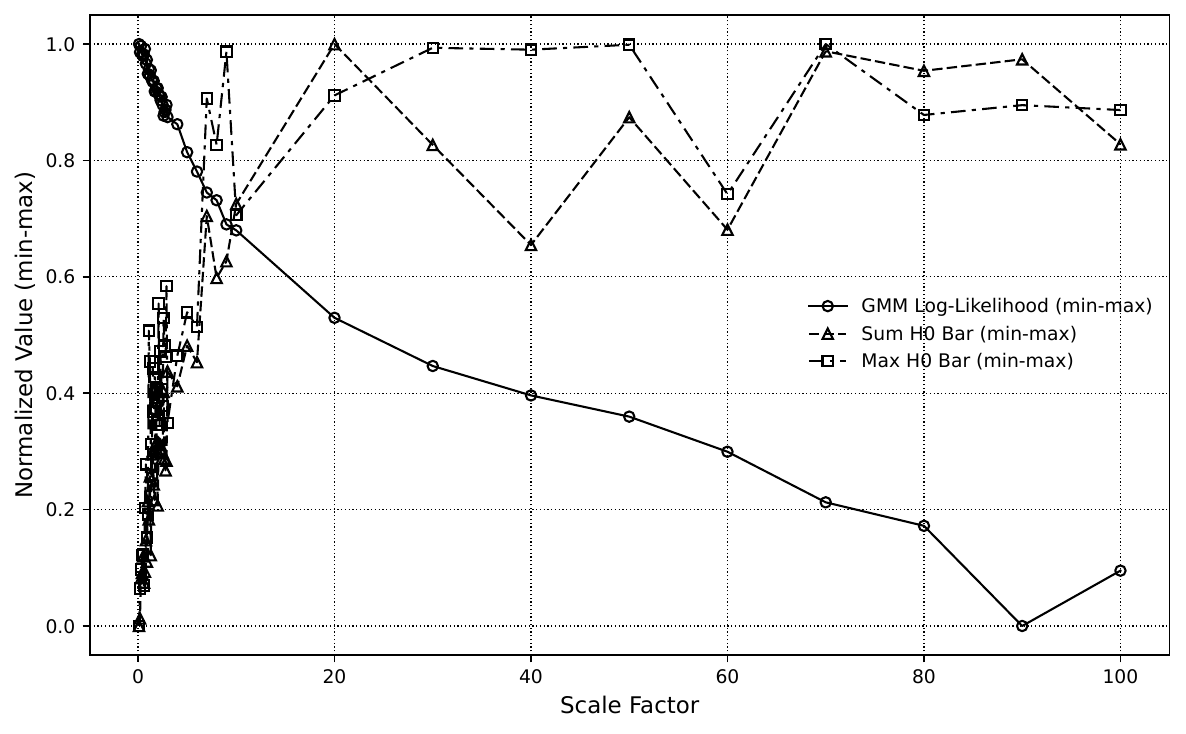}
    \caption{Scale factor vs. GMM log-likelihood and \(H_0\) persistence statistics}
    \label{fig:scale_factor_vs_metrics_h0}
\end{figure}

\begin{figure}[htbp]
    \centering
    \includegraphics[width=0.85\linewidth]{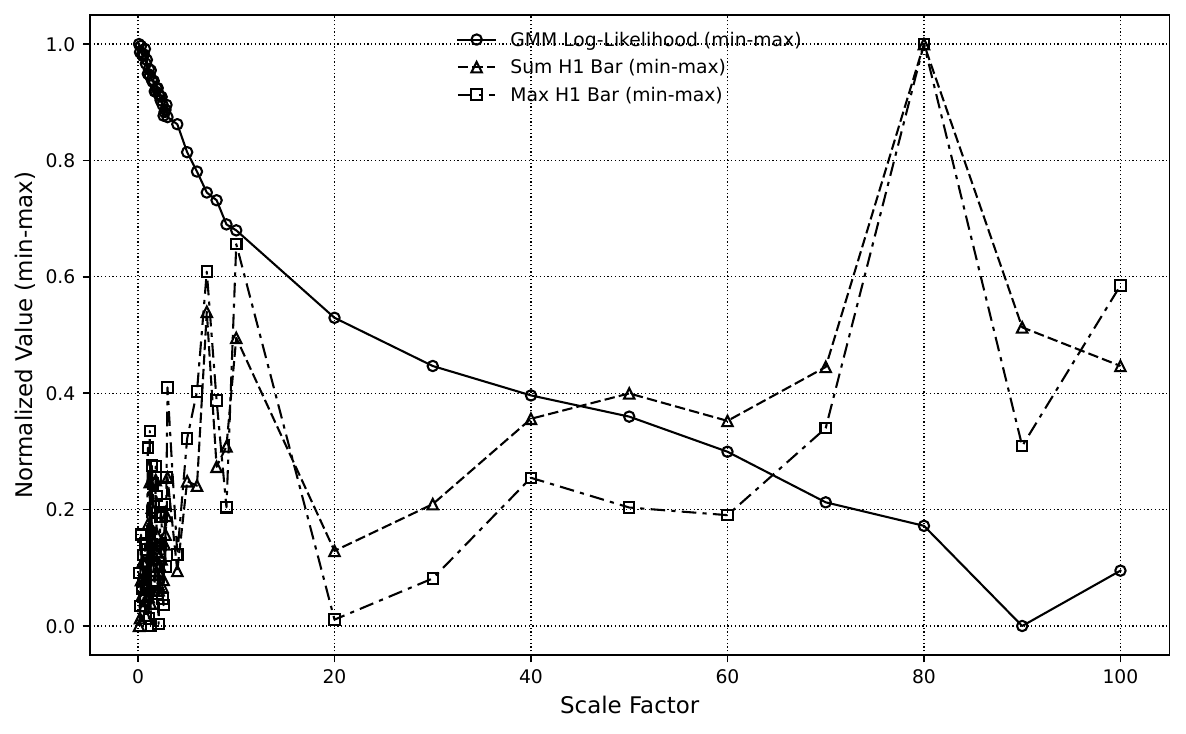}
    \caption{Scale factor vs. GMM log-likelihood and \(H_1\) persistence statistics}
    \label{fig:scale_factor_vs_metrics_h1}
\end{figure}

% ------------------------------------
\section{Change-Point Detection Methods}
\label{sec:change_point_methods}
This appendix section provides the methodological details used in Section~\ref{sec:numerical_illustration} 
to empirically assess the effectiveness of the PL+JS statistic. 
We lay out the hypothesis–testing workflow, the permutation (and Monte Carlo) procedure, 
and the assumptions required for validity, together with remarks on multiple testing control. 
While illustrative examples are drawn from Gitcoin voting data and financial time series, 
the proposed framework and testing procedure are general and applicable to a broad class of high-dimensional dynamic systems.
These materials serve as the technical reference for the empirical analyses reported in the main text.

\subsection{Hypothesis Testing Algorithm}
Let \(\mathcal{W}_1\) and \(\mathcal{W}_2\) denote two time windows, each consisting of multiple attribute vectors.
A time series of attribute vectors can be divided into multiple windows, and for each window, a persistence diagram is computed to capture the underlying topological structure.
By comparing the persistence diagrams corresponding to different windows, we can detect structural changes in the data-generating process over time.
Examples of such windows include, for instance, all vote vectors observed on each voting day in the case of DAO, or all daily return vectors within each month in the case of asset price data.
The goal is to test whether the topological structure of \(\mathcal{W}_1\) 
differs significantly from that of \(\mathcal{W}_2\).

We can compare the two persistence diagrams obtained from \(\mathcal{W}_1\) and \(\mathcal{W}_2\) using the PL+JS statistic defined in Definition~\ref{def:js_static}, and perform a permutation test (or its Monte Carlo approximation) to examine whether the attribute-vector structures in the two time windows are drawn from the same distribution. 
\begin{description}
    \item[Null hypothesis]  
    The PL+JS statistics computed from the time windows \(\mathcal{W}_1\) and \(\mathcal{W}_2\) are generated from the same distribution.
    \item[Alternative hypothesis]  
    The PL+JS statistics computed from the time windows \(\mathcal{W}_1\) and \(\mathcal{W}_2\) are generated from different distributions.
\end{description}
The detailed computation procedure by Monte Carlo is shown in Algorithm~\ref{algo:computation_p_value}.
\begin{algorithm}[htbp]
\caption{Computation of the permutation \(p\)-value}\label{algo:computation_p_value}
\begin{algorithmic}
\State \textbf{Input:} Number of attribute vectors: \(N_1\) from \(\mathcal{W}_1\), \(N_2\) from \(\mathcal{W}_2\)
\State \(V_1 \gets\) set of attribute vectors from \(\mathcal{W}_1\)
\State \(V_2 \gets\) set of attribute vectors from \(\mathcal{W}_2\)
\State \(true\_dist \gets\) PL+JS statistic from persistence diagrams of \(V_1\) and \(V_2\)
\State \(all\_vectors \gets \mathrm{concat}(V_1,V_2)\)
\State \(dist\_list \gets [\,]\) \Comment empty list
\State \(count \gets 0\)
\While{\(count <\) number of shuffles}
    \State \(U_1, U_2 \gets\) randomly select \(N_1\) and \(N_2\) vectors from \(all\_vectors\) (without replacement)
    \State \(dist \gets\) PL+JS statistic computed from persistence diagrams of \(U_1\) and \(U_2\)
    \State append \(dist\) to \(dist\_list\)
    \State \(count \gets count + 1\)
\EndWhile
\State \(p\_val \gets \dfrac{1 + |\{\,d \in dist\_list :\, d \ge true\_dist\,\}|}{\,\text{number of shuffles}+1\,}\)
\State \textbf{Output:} \(p\_val\)
\end{algorithmic}
\end{algorithm}

% ------------------------------------------------------
\subsection{Assumptions, Validity, and Error Control for Permutation Testing}
\label{sec:permutation_test}

This subsection complements the algorithmic description in Section~\ref{algo:computation_p_value}
by detailing the assumptions required for the validity of the permutation framework, 
its theoretical guarantees (conservativeness and exchangeability), 
and practical issues such as Monte Carlo accuracy and multiple-testing control.

Let \((\Omega, \mathcal{F}, \mathbb{P})\) be a probability space.  
Let  \(\mathcal{W}_1\) and  \(\mathcal{W}_2\) be any two distinct time windows. 
For \(k\in\{1,2\}\), let 
\[
  X^{(k)}_1,\ldots,X^{(k)}_{N_k}:\Omega\to\mathcal{X}
\]
be i.i.d.\ random variables sharing the same distribution \(P_k\) under \(\mathbb{P}\). 
(Here, “i.i.d.” means that the random variables are independent and identically distributed with respect to  \(\mathbb{P}\).)
Each realization of \(X_i^{(k)}\) represents one attribute vector from \(\mathcal{W}_k\).
Note that in this section we do not impose any specific assumptions on the probability distributions; in particular, we do not assume them to follow a GMM as considered in Section~\ref{sec:gmm}.
We denote by
\[
  X^{(k)} := \bigl(X^{(k)}_1,\ldots,X^{(k)}_{N_k}\bigr)\in\mathcal{X}^{N_k}
\]
the sample of attribute vectors from \(\mathcal{W}_k\).

Our aim is to test the null hypothesis
\[
  H_{\mathrm{null}} : P_1 = P_2
\]
against the alternative
\[
  H_{\mathrm{alt}} : P_1 \neq P_2.
\]
In other words, we investigate whether the distribution of attribute vectors from \(\mathcal{W}_1\) coincides with that from \(\mathcal{W}_2\).  
Rejecting the null hypothesis indicates that the collective voting behavior has significantly changed between the two time windows.  
To assess this, we employ a permutation test with Monte Carlo approximation, using the PL+JS statistic introduced in Definition~\ref{def:js_static} as a measure of discrepancy between the two empirical samples.

To mathematically formulate the permutation test, we introduce the following notation.

\begin{definition}
For 
\(
  x = (x_1, \ldots, x_{N_1}) \in \mathcal{X}^{N_1}
\)
and 
\(
  y = (y_1, \ldots, y_{N_2}) \in \mathcal{X}^{N_2},
\)
define the concatenation operator
\[
  \mathrm{concat} \colon \mathcal{X}^{N_1} \times \mathcal{X}^{N_2} 
    \to \mathcal{X}^{N_1+N_2}
\]
by
\[
  \mathrm{concat}(x, y) := (x_1, \ldots, x_{N_1}, y_1, \ldots, y_{N_2}).
\]
\end{definition}

Using this notation, the \emph{pooled sample} associated with the random vectors 
\(X^{(1)}\) and \(X^{(2)}\) is defined as
\[
  Z = (Z_1,\ldots,Z_{N_1+N_2}) 
  := \mathrm{concat}\bigl(X^{(1)}, X^{(2)}\bigr) \in \mathcal{X}^{N_1+N_2}.
\]

Let us recall the notion of exchangeability.
\begin{definition}[Exchangeability]
\label{def:exchangeability}
Let \(Z=(Z_1,\ldots,Z_{N_1+N_2})\in\mathcal{X}^{N_1+N_2}\) be random vector on the probability space \((\Omega, \mathcal{F},\mathbb{P})\). 
We say that \(Z\) is \emph{exchangeable} if for every permutation 
\(\pi\) of \(\{1,\ldots,{N_1+N_2}\}\),
\[
   (Z_1,\ldots,Z_{N_1+N_2})\ \stackrel{d}{=}\ (Z_{\pi(1)},\ldots,Z_{\pi({N_1+N_2})}),
\]
where \(\stackrel{d}{=}\) denotes equality in distribution. 
\end{definition}

To ensure that the pooled sample \(Z\) is exchangeable under the above setting, we impose the following assumption.
\begin{assumption}[independence across time windows]
The two samples are independent as random vectors:  \(\sigma\)-algebras \( \sigma\bigl[X^{(1)}\bigr]\) and  \(\sigma\bigl[X^{(2)}\bigr]\) are independent 
(equivalently,  \((X^{(1)},X^{(2)})\sim P_1^{\otimes N_1}\otimes P_2^{\otimes N_2}\)).
\end{assumption}

Under the hypothesis \(H_{\mathrm{null}}: P_1=P_2=:P\), the pooled sample \(Z\) has independent and identically distributed components with common distribution \(P\);
 equivalently, the joint distrubution of \(Z\) is
 \(P^{\otimes (N_1+N_2)}.\)
 Hence \(Z\) is exchangeable.

\begin{remark}[Independence of Vote Vectors]
\label{rem:independence_vote_vectors}
When the attribute vectors represent vote vectors (e.g., in the Gitcoin dataset), 
the independence---both among voters within each time window and between time windows---is natural, 
because each voter makes an individual decision and each voting round or day is regarded as an independent event. 
In most Web3 governance platforms (e.g., Snapshot), however, 
the ongoing voting outcomes are publicly visible during the voting period, 
which may induce weak social dependence among voters. 
Such transparency can lead to minor correlations in voting behavior, 
yet these effects are typically limited compared with the heterogeneity of individual vote vectors, 
so that the assumption of approximate independence remains reasonable for statistical testing.
\end{remark}

\begin{remark}[Independence of Financial Attribute Vectors]
\label{rem:independence_financial_attribute_vectors}
When the attribute vectors represent financial quantities such as asset prices or interest rates, 
the raw levels typically exhibit strong temporal dependence.  
To approximate independence both within and between time windows, 
it is more appropriate to use return-type or difference-type vectors.  
For asset prices, one may consider the simple return
\[
  R_t^{(i)} = \frac{S_t^{(i)} - S_{t-1}^{(i)}}{S_{t-1}^{(i)}} 
  \quad \text{or} \quad
  R_t^{(i)} = \frac{S_t^{(i)}}{S_{t-1}^{(i)}} - 1,
\]
or the log-return
\[
  r_t^{(i)} = \log S_t^{(i)} - \log S_{t-1}^{(i)}.
\]
For interest rates (or yields), which are already expressed as ratios or percentages, 
it is more appropriate to use difference-type measures such as
\[
  \Delta r_t^{(i)} = r_t^{(i)} - r_{t-1}^{(i)},
\]
representing daily or monthly changes in the rate level.  
Using these transformed vectors mitigates serial dependence 
and makes the independence assumption across samples more reasonable.
\end{remark}

Let us consider the family of index subsets
\[
  \mathcal{H} := \{\, h \subset \{1,\ldots,N_1+N_2\} :\,  |h| = N_1 \,\}.
\]
Here, \(|h|\) denotes the cardinality of the set \(h\).
Thus, \(\mathcal{H}\) consists of all subsets of \(\{1,\ldots,N_1+N_2\}\) having exactly \(N_1\) elements.
In particular, \(|\mathcal{H}| = \binom{N_1+N_2}{N_1}\).

For each \(h \in \mathcal{H}\), write \(h^c := \{1,\ldots,N_1+N_2\} \setminus h\) and define the \emph{split} of the pooled sample
\[
  Z_h := (Z_i)_{i \in h} \in \mathcal{X}^{N_1}, 
  \qquad 
  Z_{h^c} := (Z_i)_{i \in h^c} \in \mathcal{X}^{N_2}.
\]
Define the permuted statistic \(\mathcal{T}_h(Z) := T_\ell(Z_h, Z_{h^c})\).
Let \(h_0 := \{1, \dots, N_1\}\) denote the original labeling so that \(\mathcal{T}_{h_0}(z)\) is the observed statistic.

In the following, we provide the definitions of both the exact \(p\)-value of the permutation test and the Monte Carlo \(p\)-value obtained via Monte Carlo approximation. 

\begin{definition}[Exact Permutation \texorpdfstring{\(p\)}{p}-Value]
The exact permutation \(p\)-value of the permutation test is
\[
  p_{\mathrm{exact}}(Z)
  := \frac{1}{|\mathcal{H}|} \sum_{h \in \mathcal{H}}
     \mathbf{1}_{\left\{\, \mathcal{T}_h(Z) \ge \mathcal{T}_{h_0}(Z) \,\right\}}.
\]
\end{definition}

When the total number of labelings 
\(|\mathcal{H}| = \binom{N_1+N_2}{N_1}\) is too large for exhaustive enumeration, 
we approximate the exact permutation distribution by Monte Carlo.

\begin{definition}[Monte Carlo Permutation \texorpdfstring{\(p\)}{p}-Value] 
Let \(n \in \mathbb{N}\), \(n \ge 1\), denote the number of random permutations 
(i.e., labelings) to be drawn from \(\mathcal{H}\).
Take i.i.d.\ random variables \(H^{(j)}\sim\mathrm{Unif}(\mathcal H)\) (\(1\le j\le n\)) independent of \(Z\).
An Monte Carlo permutation \(p\)-value obtained via Monte Carlo approximation, with standard ``\(+1\) correction'' (to avoid zero \(p\)-values), is given by
\begin{equation}
\label{eq:approximate_p_value}
\widehat{p}_n(Z)
:= \frac{ 1 + \sum_{j=1}^{n} 
\mathbf{1}_{\left\{\, T_{H^{(j)}}(Z) \ge \mathcal{T}_{h_0}(Z) \,\right\}} }{\,n+1\,}
\end{equation}
\end{definition}

The following proposition shows that, under the null hypothesis, the permutation test is conservative:
both the exact and the Monte Carlo permutation \(p\)-values control the Type~I error at level~\(\alpha\).

\begin{proposition}[Conservativeness of Permutation \(p\)-Values]
\label{prop:permutation_test_conservative}
Assume the null hypothesis \(H_{\mathrm{null}}: P_1 = P_2\), so that the concatenated sample 
\(Z=\mathrm{concat}(X^{(1)},X^{(2)})\) is exchangeable with respect to the set of labelings \(\mathcal H\).
Let \(T\) be the test statistic, and let \(\mathcal{T}_h(Z)=T(Z_h,Z_{h^c})\) denote its value under labeling \(h\in\mathcal H\).
Let \(p_{\mathrm{exact}}(Z)\) and \(\widehat p_n(Z)\) be the exact and Monte Carlo permutation \(p\)-values, respectively, defined in the preceding paragraphs.
Then, for any \(\alpha\in(0,1)\),
\[
  \mathbb{P}\!\bigl(p_{\mathrm{exact}}(Z)\le \alpha\bigr)\ \le\ \alpha.
\]
Moreover, there exists an augmented probability space \((\hat\Omega,\hat{\mathcal F},\hat{\mathbb P})\) 
carrying i.i.d.\ random labelings \(H^{(1)},\ldots,H^{(n)}\sim\mathrm{Unif}(\mathcal H)\), 
independent of \(Z\), such that for every \(n\in\mathbb N\),
\[
  \hat{\mathbb{P}}\!\bigl(\widehat p_n(Z)\le \alpha\bigr)\ \le\ \alpha.
\]
Thus both the exact and the Monte Carlo permutation tests are conservative:
their Type~I error rate does not exceed the nominal level~\(\alpha\).
\end{proposition}

\begin{proof}[Proof sketch]
The results are standard properties of permutation tests under the null hypothesis of exchangeability. 
For any measurable test statistic \(T\), the permutation \(p\)-value 
is known to be super-uniform, that is, 
\(\Pr(p_{\mathrm{exact}}\le\alpha)\le\alpha\), 
and the same holds for its Monte Carlo approximation. 
A concise treatment can be found in 
\cite{lehmann2005} (Chapter~15) 
and \cite{phipson2010}. 
\end{proof}

To assess the adequacy of the number \(n\) of random labelings, we record a bound
on the Monte Carlo variability of the permutation \(p\)-value.

\begin{proposition}[Monte Carlo Variance Bound]
\label{prop:mc-se}
Under the setting described above, we have
\[
  \hat{\mathbb{V}}\!\bigl(\widehat p_n(Z)\,\big|\,Z\bigr)
  \;\le\;
  \frac{1}{4(n+1)}
  \quad\text{a.s.},
\]
where \(\hat{\mathbb V}(\cdot\,|\,Z)\) denotes the conditional variance under \(\hat{\mathbb P}\) given \(Z\).
\end{proposition}

\begin{proof}
Define
\[
  p(Z)\;:=\;\hat{\mathbb{P}}\!\bigl(T_{H^{(j)}}(Z)\ge \mathcal{T}_{h_0}(Z)\,\big|\,Z\bigr),
\]
which does not depend on \(j\).
Conditional on \(Z\), the indicators 
\(I_j(Z)=\mathbf 1_{\{\,T_{H^{(j)}}(Z)\ge \mathcal{T}_{h_0}(Z)\,\}}\)
are i.i.d.\ \(\mathrm{Bernoulli}\!\bigl(p(Z)\bigr)\), 
so that 
\(\sum_{j=1}^n I_j(Z)\sim\mathrm{Binomial}\!\bigl(n,p(Z)\bigr)\).
Therefore,
\[
  \hat{\mathbb{V}}\!\bigl[\widehat p_n(Z)\mid Z\bigr]
  =\frac{n\,p(Z)\bigl(1-p(Z)\bigr)}{(n+1)^2}
  \;\le\;\frac{1}{4(n+1)}.
\]
\end{proof}

\begin{remark}[Multiple Testing Across Consecutive Windows]
\label{rem:multiple_testing}
In practical applications, it is often of interest to test whether 
a sequence of consecutive windows 
\(\mathcal{W}_1, \ldots, \mathcal{W}_K\) 
share a common distribution.  
This involves testing the series of local null hypotheses
\[
  H_{\mathrm{null},k}: P_k = P_{k+1},
  \qquad k = 1, \ldots, K-1,
\]
simultaneously at a given significance level~\(\alpha\).
When each test is conducted at level~\(\alpha\), 
the overall probability of at least one false rejection, 
known as the \emph{family-wise error rate} (FWER),
tends to exceed~\(\alpha\),
making the overall testing procedure \emph{anti-conservative} 
(i.e., prone to false positives).

A simple and widely used approach to control the FWER is 
the \emph{Bonferroni correction}, 
which replaces the nominal level~\(\alpha\) by \(\alpha/(K-1)\) 
for each of the \(K-1\) individual tests.
Although the Bonferroni method guarantees that the FWER does not exceed~\(\alpha\),
it can be overly conservative and thus reduce statistical power.

To improve power while maintaining FWER control, 
the \emph{Holm step-down procedure} can be applied.
Let \(p_{(1)} \le \cdots \le p_{(K-1)}\) be the ordered sequence of individual \(p\)-values.
Starting from the smallest one, 
the procedure sequentially rejects \(H_{\mathrm{null},(1)}, H_{\mathrm{null},(2)}, \ldots\)
as long as \(p_{(k)} \le \alpha/(K-k)\) holds,
and stops at the first violation of this inequality.
All previously rejected hypotheses are then declared significant.
This step-down method guarantees FWER control at level~\(\alpha\)
and typically yields higher power than Bonferroni’s rule.

An alternative strategy is to replace the local hypotheses above
by a single \emph{global null hypothesis},
\[
  H_{\mathrm{null}}^{\mathrm{global}} : P_1 = P_2 = \cdots = P_K,
\]
and perform a \emph{max-\(T\)} test.
In this approach, permutation resampling is carried out 
over the pooled sample \(\mathcal{W}_1 \cup \cdots \cup \mathcal{W}_K\),
redistributing the data across windows while keeping their sizes fixed.
Because the test statistic must be computed for all window pairs
under many permutations, the computational cost increases substantially,
especially when persistence diagrams are involved.

From a practical viewpoint, 
when the goal is to promptly detect distributional changes over time, 
the Bonferroni and Holm step-down procedures 
often provide a good balance between statistical validity and computational efficiency.
\end{remark}

\end{appendices}

% \bibliographystyle{sn-mathphys-num}
\bibliography{sn-bibliography}% common bib file
% \nocite{*}

% \input sn-article.bbl

\end{document}